\documentclass{amsart}
\address{\newline{\normalsize University of Edinburgh, Kings Buildings, Mayfield
Road, Edinburgh EH9 3JZ, UK}\newline{\it E-mail address}:
ilkarjem@rambler.ru}
\usepackage{amscd,amssymb}
\usepackage{amsthm,amsmath,amssymb}
\usepackage[dvips]{graphicx}
\usepackage[matrix,arrow]{xy}

\newtheorem{theorem}[equation]{Theorem}
\newtheorem{proposition}[equation]{Proposition}
\newtheorem{lemma}[equation]{Lemma}
\newtheorem{corollary}[equation]{Corollary}

\theoremstyle{definition}
\newtheorem{example}[equation]{Example}

\theoremstyle{remark}
\newtheorem{remark}[equation]{Remark}

\makeatletter\@addtoreset{equation}{section}\makeatother

\renewcommand{\labelenumi}{(\theenumi)}
\renewcommand{\theenumi}{\arabic{enumi}}

\begin{document}


\title{Fano threefolds with canonical Gorenstein singularities and big degree}

\author{Ilya Karzhemanov}

\thanks{The work was partially supported by
RFFI grant No. 08-01-00395-a and grant N.Sh.-1987.2008.1.}

\sloppy

\begin{abstract}
We provide a complete classification of Fano threefolds $X$ having
canonical Gorenstein singularities and anticanonical degree
$(-K_{X})^{3} = 64$.
\end{abstract}

\maketitle

\section{Introduction}
\label{section:introduction}
\renewcommand{\thefootnote}{\arabic{footnote})}

Let $X$ be a Fano threefold over $\mathbb{C}$.\footnote{Throughout
the paper we assume all Fano threefolds to have at worst canonical
Gorenstein singularities.} Then for the anticanonical degree
$(-K_{X})^{3}$ of $X$, the following result takes place:

\begin{theorem}[see {\cite[Theorem 1.5]{Prokhorov-degree}}]
\label{theorem:prokhorov-degree} The estimate $(-K_{X})^{3}
\leqslant 72$ holds. Moreover, $(-K_{X})^{3} = 72$ iff either $X =
\mathbb{P}(3,1,1,1)$ or $X = \mathbb{P}(6,4,1,1)$.
\end{theorem}

Thus Theorem~\ref{theorem:prokhorov-degree} gives the exact bound
on $(-K_{X})^{3}$. This makes one expect to classify all Fano
threefolds in every class of fixed degree. In fact, for
$(-K_{X})^{3}
> 64$, one gets

\begin{theorem}[see {\cite[Theorem 1.5]{karz}}]
\label{theorem:main-0} If $64 < (-K_{X})^{3} < 72$, then $X$ is
one of the following:

\begin{itemize}

\item $X_{70}$: the image of the anticanonically embedded
threefold $\mathbb{P}(6,4,1,1) \subset \mathbb{P}^{38}$ under
birational linear projection from a singular $\mathrm{cDV}$ point
on $\mathbb{P}(6,4,1,1)$. In this case, $(-K_{X})^{3} = 70$ and
the singularities of $X$ are worse than $\mathrm{cDV}$;

\item $X_{66}$: the anticanonical image of the
$\mathbb{P}^2$-bundle
$\mathbb{P}(\mathcal{O}_{\mathbb{P}^{1}}(5)\oplus\mathcal{O}_{\mathbb{P}^{1}}(2)\oplus\mathcal{O}_{\mathbb{P}^{1}})$.
In this case, $(-K_{X})^{3} = 66$ and the singularities of $X$ are
worse than $\mathrm{cDV}$.

\end{itemize}

\end{theorem}

The aim of the present paper is to classify those Fano threefolds
$X$ which have $(-K_{X})^{3} = 64$ (compare with \cite{jah-rad-1},
\cite{jah-rad-2} and \cite{jah-rad-3} for related results on the
classification of smooth threefolds with big and nef anticanonical
class). Note that for smooth $X$, the only possibility is $X =
\mathbb{P}^{3}$ (see \cite{VA-1}, \cite{VA-2}). There are more
examples however in the singular case:

\begin{example}
\label{example:examp-0} Let $X \subset \mathbb{P}^{9}$ be the cone
over the anticanonically embedded surface $S: = \mathbb{P}^{1}
\times \mathbb{P}^{1}$. Let $f: Y \longrightarrow X$ be the blowup
of the vertex on $X$. Then $Y = \mathbb{P}(\mathcal{O}_{S} \oplus
\mathcal{O}_{S}(-K_{S}))$ and $f$ is the birational contraction of
the negative section of the $\mathbb{P}^{1}$-bundle $Y$. From the
relative Euler exact sequence we obtain $-K_{Y} \sim 2M$, where
$\mathcal{O}_{Y}(M) \simeq \mathcal{O}_{Y}(1)$ is the tautological
sheaf on $Y$ (see \cite[Proposition 4.26]{Mori-Mukai}). On the
other hand, $f$ is given by the linear system $|M|$, which implies
that $K_{Y} = f^{*}(K_{X})$. In particular, $X$ is a Fano
threefold. Furthermore, from the Hirsch formula (see
\cite{grothendieck-chern-classes}) we deduce that $(-K_{X})^{3} =
(-K_{Y})^{3} = 64$. Note also that $X$ is toric.
\end{example}

\begin{example}
\label{example:examp-4} Let $X \subset \mathbb{P}^{9}$ be the cone
over the anticanonically embedded surface $S: = \mathbb{F}_{1}$.
Let $f: Y \longrightarrow X$ be the blowup of the vertex on $X$.
Then $Y = \mathbb{P}(\mathcal{O}_{S} \oplus
\mathcal{O}_{S}(-K_{S}))$ and $f$ is the birational contraction of
the negative section of the $\mathbb{P}^{1}$-bundle $Y$. As in
Example~\ref{example:examp-0}, we get $K_{Y} = f^{*}(K_{X})$ and
$(-K_{Y})^{3} = 64$, which implies that $X$ is a Fano threefold
with $(-K_{X})^{3} = 64$. Note again that $X$ is toric.
\end{example}

\begin{remark}
\label{remark:about-terminal-gorenstein-fano} It is easy to see
that singularities of Fano threefolds in
Examples~\ref{example:examp-0}, \ref{example:examp-4} are worse
than terminal. In fact, as we will see below (cf.
Remark~\ref{remark:namikawa-smoothing-rem}), projective space
$\mathbb{P}^{3}$ is the only Fano threefold among those that have
at worst terminal Gorenstein singularities and anticanonical
degree equal $64$.
\end{remark}

We now state the main result of the present paper:

\begin{theorem}
\label{theorem:main-1} Let $X$ be a Fano threefold. If
$(-K_{X})^{3} = 64$, then $X$ is one of the following:

\begin{enumerate}

\item $\mathbb{P}^{3}$;

\item\label{num-0} the cone from
Example~\ref{example:examp-0};

\item\label{num-3} the cone from
Example~\ref{example:examp-4};

\item\label{num-1} the image of the anticanonically embedded threefold
$\mathbb{P}(3,1,1,1) \subset \mathbb{P}^{38}$ under birational
linear projection from the tangent space at a smooth point on
$\mathbb{P}(3,1,1,1)$;

\item\label{num-4} the image of the anticanonically embedded threefold
$\mathbb{P}(6,4,1,1) \subset \mathbb{P}^{38}$ under birational
linear projection from the tangent space at a smooth point on
$\mathbb{P}(6,4,1,1)$;

\item\label{num-5} the image of the anticanonically embedded threefold $X_{70}
\subset \mathbb{P}^{37}$ under birational linear projection from a
plane;

\item\label{num-2} the image of the anticanonically embedded threefold $X_{66}
\subset \mathbb{P}^{35}$ under birational linear projection from a
singular $\mathrm{cDV}$ point on $X_{66}$.

\end{enumerate}
Moreover, all cases $(\ref{num-1})-(\ref{num-2})$ do occur, and
the singularities of $X$ in $(\ref{num-0})-(\ref{num-2})$ are
worse than $\mathrm{cDV}$.

\end{theorem}

Thus Theorems~\ref{theorem:prokhorov-degree}, \ref{theorem:main-0}
and \ref{theorem:main-1} describe all the Fano threefolds $X$ with
$(-K_{X}^{3}) \geqslant 64$.

\begin{remark}
\label{remark:small-comments} Except for $\mathbb{P}^{3}$ and the
cones from Examples~\ref{example:examp-0}, \ref{example:examp-4},
some of the threefolds from Theorem~\ref{theorem:main-1} can be
isomorphic, since there are exactly five toric Fano threefolds
which have anticanonical degree equal $64$ (see
\cite{Kreuzer-Skarke}). Yet we only restrict ourselves to the
description of all possible deformation types of Fano threefolds
in question and postpone the study of their isomorphism classes
until some other time.
\end{remark}

The proof of Theorem~\ref{theorem:main-1} relies on the methods,
developed in \cite{Prokhorov-degree}, \cite{karz} to prove
Theorems~\ref{theorem:prokhorov-degree}, \ref{theorem:main-0},
respectively. Namely, one starts with a birational contraction $f
: Y \longrightarrow X$, a \emph{terminal $\mathbb{Q}$-factorial
modification of $X$} (see Section~\ref{section:preliminaries}
below), where $Y$ has at worst terminal factorial singularities
and $f$ is birational with $K_Y \sim f^{*}(K_{X})$ (cf.
Examples~\ref{example:examp-0}, \ref{example:examp-4}). Then the
proof goes via the step-by-step analysis of all possible
$K_Y$-negative extremal contractions on $Y$ (see
Sections~\ref{section:mori-fiber-space}--\ref{section:contraction-to-surface-case}
below). However, this requires a quite detailed study of the
geometry of threefolds from
Theorems~\ref{theorem:prokhorov-degree} and \ref{theorem:main-0}
(see Sections~\ref{section:preliminaries},
\ref{section:auxiliary}) and considering a lot of intermediate
cases, which makes it hopeless to get any reasonable
classification of Fano threefolds $X$ with $(-K_{X})^{3} < 64$.

\bigskip

I would like to thank I. A. Cheltsov and Yu. G. Prokhorov for
helpful discussions and for initiating my work on the problem.
Finally, I am grateful to the referee, whose comments and careful
reading have helped me to improve the exposition.

\bigskip

\section{Notation and conventions}
\label{section:notation-and-conventions}

We use standard notions and facts from the theory of minimal
models and singularities of pairs (see \cite{Kollar-Mori},
\cite{kawamata-matsuda-matsuki}, \cite{kollar-sing-of-pairs}). We
also use standard notions and facts from the theory of varieties
and schemes (see \cite{hartshorne-ag}, \cite{griff-harr}). All
varieties are assumed to be algebraic and defined over
$\mathbb{C}$. Throughout the paper we use standard notions and
notation from \cite{Kollar-Mori}, \cite{kawamata-matsuda-matsuki},
\cite{hartshorne-ag}. However, let us introduce some more:

\begin{itemize}

\item We denote by $\mathrm{Sing}(V)$ the singular locus of an algebraic variety $V$.
We denote by $(O \in V)$ the analytic germ of a point
$O\in\mathrm{Sing}(V)$.

\smallskip

\item For a $\mathbb{Q}$-Cartier divisor $M$ (resp. a linear
system $\mathcal{M}$) and an algebraic cycle $Z$ on a projective
normal variety $V$, we denote by $M \big \vert_{Z}$ (resp.
$\mathcal{M} \big \vert_{Z}$) the restriction of $M$ (resp.
$\mathcal{M}$) to $Z$. We denote by $Z_{1} \cdot \ldots \cdot
Z_{k}$ the intersection of algebraic cycles $Z_{1}, \ldots,
Z_{k},k \in \mathbb{N}$, in the Chow ring of $V$.

\smallskip

\item $M_{1}
\equiv M_{2}$ (resp. $Z_{1} \equiv Z_{2}$) stands for the
numerical equivalence of two $\mathbb{Q}$-Cartier divisors
$M_{1},M_{2}$ (resp. two algebraic $1$-cycles $Z_{1},Z_{2}$) on a
normal projective variety $V$. We denote by $N_{1}(V)$ the group
of classes of algebraic cycles on $V$ modulo numerical equivalence
(sometimes the numerical class of a cycle $Z$ on $V$ will be
denoted by $[Z]$). We denote by $\rho(V)$ the Picard number of
$V$. $D_1 \sim D_2$ stands for the the linear equivalence of two
Weil divisors $D_1,D_2$ on $V$. We denote by $\mathrm{Pic}(V)$ the
group of Cartier divisors on $V$ modulo linear equivalence.

\smallskip

\item Normal three-dimensional variety $V$ is called a \emph{Fano threefold}
if it has at worst canonical Gorenstein singularities and the
anticanonical divisor $-K_{V}$ is ample. Normal projective
three-dimensional variety $V$ is called a \emph{weak Fano
threefold} if it has at worst canonical singularities and the
anticanonical divisor $-K_{V}$ is nef and big. The number
$(-K_{V})^{3}$ is called the (anticanonical) \emph{degree} of $V$.
For a Fano threefold $V$, any curve $Z \subset V$ satisfying $Z
\simeq \mathbb{P}^1$ and $-K_{V} \cdot Z = 1$ is called a
\emph{line}. Similarly, any surface $\Pi \subset V$ satisfying
$\Pi \simeq \mathbb{P}^2$ and $-K_{V}^2 \cdot \Pi = 1$ is called a
\emph{plane}.

\smallskip

\item For a Weil divisor $D$ on a normal variety
$V$, we denote by $\mathcal{O}_{V}(D)$ the corresponding
divisorial sheaf on $V$ (sometimes we denote both by
$\mathcal{O}_{V}(D)$ or by $D$).

\smallskip

\item For a coherent sheaf $F$ on a normal projective variety $V$, we denote by
$H^{i}(V, F)$ the $i$-th cohomology group of $F$. We set $h^{i}(V,
F) := \dim H^{i}(V, F)$ and $\chi(V, F) := \displaystyle\sum_{i =
1}^{\dim V}(-1)^i h^{i}(V, F)$. We also denote by $c_{i}(F)$ the
$i$-th Chern class of $F$.

\smallskip

\item For a vector bundle $E$ on a smooth variety
$V$, we denote by $\mathbb{P}_{V}(E)$ (or simply by
$\mathbb{P}(E)$ if no confusion is likely) the associated
projective bundle.

\smallskip

\item For a Cartier divisor $M$ on a normal projective variety $V$, we denote by $|M|$ the corresponding complete linear system on $V$.
For an algebraic cycle $Z$ on $V$, we denote by $|M - Z|$ the
linear subsystem in $|M|$ consisting of all divisors passing
through $Z$. For a linear system $\mathcal{M}$ on $V$, we denote
by $\mathrm{Bs}\,(\mathcal{M})$ the base locus of $\mathcal{M}$.
If $\mathcal{M}$ does not have fixed components, we denote by
$\Phi_{\mathcal{M}}$ the corresponding rational map.

\smallskip

\item For a birational map $\psi: V' \dashrightarrow V$ between normal projective varieties
and an algebraic cycle $Z$ (resp. a linear system $\mathcal{M}$)
on $V$, we denote by $\psi_{*}^{-1}(Z)$ (resp. by
$\psi_{*}^{-1}(\mathcal{M})$) the proper transform of $Z$ (resp.
of $\mathcal{M}$) on $V'$.

\smallskip

\item For a rational map $\chi$ from an algebraic variety $V$ and a
subvariety $Z \subset V$, we denote by $\chi\big\vert_{Z}$ the
restriction of $\chi$ to $Z$.

\smallskip

\item We denote by $\kappa(V)$ the Kodaira dimension of a normal projective variety
$V$.

\smallskip

\item We denote by $\mathbb{F}_{n}$ the Hirzebruch
surface with a fiber $l$ and the minimal section $h$ of the
natural projection $p_n : \mathbb{F}_{n} \to \mathbb{P}^1$ so that
$(h^{2}) = -n,n \in \mathbb{Z}_{\geqslant 0}$.

\end{itemize}

\bigskip

\section{Preliminaries}
\label{section:preliminaries}

Let $X$ be a Fano threefold. From the Riemann-Roch formula and
Kawamata-Viehweg vanishing theorem we obtain the equality
\begin{equation}
\nonumber \dim |-K_{X}| = -\frac{1}{2}K_{X}^{3} + 2.
\end{equation}
The number $g:= -\displaystyle\frac{1}{2}K_{X}^{3} + 1$ is an
integer and is called the \emph{genus} of $X$. We may write
\begin{equation}
\label{deg-estimate} \dim |-K_{X}| = g + 1 \ \mbox{and} \
(-K_{X})^{3} = 2g-2.
\end{equation}

The following results illustrate the behavior of the anticanonical
linear system $|-K_{X}|$ with respect to the estimates on the
degree $(-K_{X})^{3}$:

\begin{theorem}[see \cite{Jahnke-Radloff}]
\label{theorem:non-free-antican-system} If $\mathrm{Bs}|-K_{X}|
\ne \emptyset$, then $(-K_{X})^{3} \leqslant 22$.
\end{theorem}

\begin{theorem}[see {\cite[Theorem 1.5]{CPS}}]
\label{theorem:free-antican-system-1} If $\mathrm{Bs}|-K_{X}| =
\emptyset$ and the morphism $\Phi_{\scriptscriptstyle|-K_{X}|}$ is
not an embedding, then $(-K_{X})^{3} \leqslant 40$.
\end{theorem}

\begin{theorem}[see {\cite[Theorem 1.6]{CPS}}]
\label{theorem:free-antican-system-2} If the morphism
$\Phi_{\scriptscriptstyle|-K_{X}|}$ is an embedding and the image
$\Phi_{\scriptscriptstyle|-K_{X}|}(X) \subset \mathbb{P}^{g + 1}$
is not an intersection of quadrics, then $(-K_{X})^{3} \leqslant
54$.
\end{theorem}

From Theorems~\ref{theorem:non-free-antican-system},
\ref{theorem:free-antican-system-1} and
\ref{theorem:free-antican-system-2} one immediately gets

\begin{corollary}
\label{theorem:embed-as-intersection-of-quadrics} If $(-K_{X})^{3}
\geqslant 64$, then the morphism
$\Phi_{\scriptscriptstyle|-K_{X}|}$ is an embedding and the image
$\Phi_{\scriptscriptstyle|-K_{X}|}(X) \subset \mathbb{P}^{g + 1}$
is an intersection of quadrics.
\end{corollary}

One can ``simplify'' the singularities of $X$ via the next

\begin{proposition}[{see \cite[6.3]{Kollar-Mori}}]
\label{theorem:terminal-modification} There exist a normal
algebraic threefold $Y$ with at worst terminal
$\mathbb{Q}$-factorial singularities and a birational morphism $f
: Y \longrightarrow X$ such that $K_{Y} \sim f^{*}(K_{X})$.
\end{proposition}

\begin{remark}
\label{remark:K-trivial-contraction-1} Threefold $Y$ (or the
morphism $f$) from Proposition~\ref{theorem:terminal-modification}
is called a \emph{terminal $\mathbb{Q}$-factorial modification} of
$X$. Note that $Y$ is a weak Fano threefold with at worst terminal
Gorenstein $\mathbb{Q}$-factorial singularities and such that
$(-K_{X})^{3} = (-K_{Y})^{3}$. In particular, according to
\cite[Lemma 5.1]{Kawamata} $Y$ is factorial. Conversely, for every
weak Fano threefold $Y$ with at worst terminal factorial
singularities its image $X := f(Y)$ under the morphism $f :=
\varPhi_{\scriptscriptstyle|-nK_{Y}|}$ for some large $n \in
\mathbb{N}$ is a Fano threefold such that $K_{Y} \sim
f^{*}(K_{X})$ (see \cite{Kawamata}). Finally, since one may
actually assume $-K_X$ to be just nef and big in
Proposition~\ref{theorem:terminal-modification}, the image $X =
\varPhi_{\scriptscriptstyle|-nK_{Y}|}(Y)$ still has at worst
canonical Gorenstein singularities, provided such are the
singularities of $Y$.
\end{remark}

\begin{remark}
\label{remark:terminal-modifications-are-connected-by-flops} Let
$Y$, $Y'$ be two different terminal $\mathbb{Q}$-factorial
modifications of $X$. Then, since $Y$ and $Y'$ are relative
minimal models over $X$ (see the proof of the above
Proposition~\ref{theorem:terminal-modification} in
\cite[6.3]{Kollar-Mori}), by \cite[Theorem 4.3]{Kollar-flops} the
induced birational map $Y \dashrightarrow Y'$ is either an
isomorphism or a composition of $K_{Y}$-flops over $X$. In
particular, by \cite[Theorem 2.4]{Kollar-flops}, if $Y$ is smooth,
then so is $Y'$.
\end{remark}

\begin{proposition}[see {\cite[Lemmas 4.2, 4.3]{Prokhorov-degree}}]
\label{theorem:extremal-rays-cone} In the notation of
Proposition~\ref{theorem:terminal-modification}, the Mori cone
$\overline{NE}(Y)$ of $Y$ is polyhedral and generated by
contractible extremal rays $R_{i}$ (i.e., such $R_i$ that there
exists a morphism $f_{R_{i}}: Y \longrightarrow Y'$ onto a normal
algebraic threefold $Y'$, for which a curve $Z \subset Y$ is
mapped by $f_{R_{i}}$ to a point iff $[Z] \in R_{i}$).
\end{proposition}

\begin{remark}
\label{remark:K-trivial-contraction-2} Suppose that $X \ne Y$.
Then Proposition~\ref{theorem:extremal-rays-cone} implies that
divisor $K_{Y}$ determines a face of the cone $\overline{NE}(Y)$
and $f: Y \longrightarrow X$ is the extremal contraction of this
face. In particular, $X$ and $f$ are uniquely determined by $Y$
and do not depend on $n$ (cf.
Remark~\ref{remark:K-trivial-contraction-1}). We also have
$\rho(Y)
> \rho(X)$ in this case.
\end{remark}

Let us consider some examples of Fano threefolds and their
terminal $\mathbb{Q}$-factorial modifications. Firstly, birational
contractions $f: Y \longrightarrow X$ from
Examples~\ref{example:examp-0}, \ref{example:examp-4} are
obviously of this sort. Further, we have:

\begin{example}[{cf. \cite[Example 2.11]{karz}}]
\label{example:examp-1} Let $X \subset \mathbb{P}^{10}$ be the
cone over the anticanonically embedded surface $\mathbb{P}^{2}$.
Let $f: Y \longrightarrow X$ be the blowup of the vertex on $X$.
Then $Y =
\mathbb{P}(\mathcal{O}_{\mathbb{P}^{2}}\oplus\mathcal{O}_{\mathbb{P}^{2}}(3))$
and $f$ is the birational contraction of the negative section of
the $\mathbb{P}^{1}$-bundle $Y$. From the relative Euler exact
sequence we obtain $-K_{Y} \sim 2M$, where $\mathcal{O}_{Y}(M)
\simeq \mathcal{O}_{Y}(1)$ is the tautological sheaf on $Y$ (see
\cite[Proposition 4.26]{Mori-Mukai}). On the other hand, $f$ is
given by the linear system $|M|$, which implies that $Y$ is a weak
Fano threefold with $K_{Y} = f^{*}(K_{X})$. In particular, $X$ is
a Fano threefold and $f: Y \longrightarrow X$ is its terminal
$\mathbb{Q}$-factorial modification (cf.
Remark~\ref{remark:K-trivial-contraction-2}). Furthermore, it
follows from the proof of the above
Theorem~\ref{theorem:prokhorov-degree} in \cite{Prokhorov-degree}
that $X \simeq \mathbb{P}(3,1,1,1)$. It also follows easily from
\cite{Dolgachev} that $(-K_{X})^3 = 72$ (cf.
Theorem~\ref{theorem:prokhorov-degree}).
\end{example}

\begin{remark}
\label{remark:unique-terminal-modification-1} In the notation of
Example~\ref{example:examp-1}, morphism $f$ is an extremal
birational contraction with exceptional locus isomorphic to
$\mathbb{P}^{2}$. Hence there are no small $K_{Y}$-trivial
extremal contractions on $Y$. Then it follows from
Remark~\ref{remark:terminal-modifications-are-connected-by-flops}
that every terminal $\mathbb{Q}$-factorial modification of
$\mathbb{P}(3,1,1,1)$ is isomorphic to $Y$.
\end{remark}

\begin{example}[{cf. \cite[Example 2.13]{karz}}]
\label{example:examp-2} Consider the weighted projective space $X:
= \mathbb{P}(6,4,1,1)$. The singular locus of $X$ is a curve $L
\simeq \mathbb{P}^{1}$ such that for some points $P$ and $Q$ on
$L$ the germs $(P \in X)$ and $(Q \in X)$ are singularities of
type $\displaystyle\frac{1}{6}(4,1,1)$ and
$\displaystyle\frac{1}{4}(2,1,1)$, respectively, and for every
point $O \in L \setminus\{P, Q\}$ singularity $(O \in X)$ is
locally analytically isomorphic to $((0, o) \in \mathbb{C} \times
U)$, where $(o \in U)$ is the singularity of type
$\displaystyle\frac{1}{2}(1,1)$ (see \cite[5.15]{Iano-Fletcher}).
Hence the singularities of $X$ are canonical and Gorenstein (see
\cite[Theorem 3.1]{Reid-canonical-threefolds} and \cite[Remark
3.2]{Reid-canonical-threefolds}). On the other hand, we have
$\mathcal{O}_{X}(-K_{X}) \simeq \mathcal{O}_{X}(12)$ (see
\cite[Theorem 3.3.4]{Dolgachev}), which implies that divisor
$-K_{X}$ is ample. Thus $X$ is a Fano threefold. It also follows
easily from \cite{Dolgachev} that $(-K_{X})^3 = 72$ (cf.
Theorem~\ref{theorem:prokhorov-degree}).

There is a toric resolution of singularities of $X$ as follows.
First, one blows up the point $P$ with weights
$\displaystyle\frac{1}{6}(4,1,1)$, replacing $P$ this way by a
point of type $\displaystyle\frac{1}{4}(2,1,1)$. The next two
blowups, both with weights $\displaystyle\frac{1}{4}(2,1,1)$, of
the new threefold make every singularity on the preimage of $L$
look like $((0, o) \in \mathbb{C} \times U)$ as above. Finally,
blowing up the preimage of $L$ yields the toric resolution $f: Y
\longrightarrow X$. Notice that $K_{Y} = f^{*}(K_{X})$ by
construction and so $Y$ is a terminal $\mathbb{Q}$-factorial
modification of $X$.
\end{example}

\begin{remark}
\label{remark:unique-terminal-modification-2} In the notation of
Example~\ref{example:examp-2}, morphism $f$ is a composition of
extremal birational contractions (cf.
Remark~\ref{remark:K-trivial-contraction-2}), and exceptional
locus of $f$ has pure codimension $1$. Hence there are no small
$K_{Y}$-trivial extremal contractions on $Y$. Then it follows from
Remark~\ref{remark:terminal-modifications-are-connected-by-flops}
that every terminal $\mathbb{Q}$-factorial modification of
$\mathbb{P}(6,4,1,1)$ is isomorphic to $Y$ (uniqueness of $Y$ also
easily follows form the fact that $Y$ is toric).
\end{remark}

In order to construct more examples let us introduce several
results on linear projections of Fano threefolds. Firstly, we will
need the following:

\begin{theorem}[see \cite{Reid-morphisms-Kawamata}]
\label{theorem:elefant} Let $X$ be a Fano threefold. Then generic
surface in $|-K_{X}|$ has at worst Du Val singularities.
\end{theorem}

Further, let $X'$ be a Fano threefold of genus $g'$ such that the
morphism $\Phi_{\scriptscriptstyle|-K_{X'}|}$ is an embedding. Let
us identify $X'$ with its image
$\Phi_{\scriptscriptstyle|-K_{X'}|}(X')$ in $\mathbb{P}^{g' + 1}$
(cf. \eqref{deg-estimate}). Let us also assume that $X' \subset
\mathbb{P}^{g' + 1}$ is an intersection of quadrics (it follows
from Theorems~\ref{theorem:prokhorov-degree}, \ref{theorem:main-0}
(see also Examples~\ref{example:examp-0}, \ref{example:examp-4})
and Corollary~\ref{theorem:embed-as-intersection-of-quadrics} that
such $X'$ does exist). Under these assumptions we get the
following four Lemmas:

\begin{lemma}
\label{theorem:bir-proj} Let $\pi: X' \dashrightarrow X$ be the
linear projection from a subspace $\Lambda\subset\mathbb{P}^{g' +
1}$ such that $\Lambda\cap X'\ne\emptyset$ and $\dim X = 3$. Then
the map $\pi$ is birational. Conversely, if $\pi$ is birational
and $X$ is an anticanonically embedded Fano threefold, then
$\Lambda\cap X'\ne\emptyset$ and $(-K_{X'})^3 > (-K_X)^3$.
\end{lemma}

\begin{proof}
The first assertion is obvious because $X'$ is an intersection of
quadrics by assumption. For the second assertion, observe that $g'
> g$ (for linear projections decrease the dimensions of ambient
projective spaces). In particular, given that $\Lambda\cap X' =
\emptyset$, we obtain $(-K_{X'})^3 = (-K_X)^3$ by projection
formula, which contradicts \eqref{deg-estimate}.
\end{proof}

\begin{lemma}
\label{theorem:only-cdv-points} In the same setting as in
Lemma~\ref{theorem:bir-proj}, if $X$ is an anticanonically
embedded Fano threefold, then the set
$\Lambda\cap\mathrm{Sing}(X')$ consists only of $\mathrm{cDV}$
points on $X'$.
\end{lemma}

\begin{proof}
Pick some point $P\in\Lambda \cap \mathrm{Sing}(X')$. Suppose that
singularity $(P \in X')$ is worse than $\mathrm{cDV}$. Then
singularities of a generic normal surface $S' \in |-K_{X'}|$
passing through $P$ are worse than Du Val. In particular, we have
$\kappa(S') < 0$. On the other hand, since $-K_X \sim
\mathcal{O}_{X}(1)$ by assumption, $\pi$ gives a birational map
$S' \dashrightarrow S$ onto a general surface $S\in |-K_{X}|$,
which implies that $\kappa(S) = \kappa(S') < 0$. But by
Theorem~\ref{theorem:elefant}, $S$ is a $\mathrm{K3}$ surface with
at worst Du Val singularities, hence with $\kappa(S) = 0$, a
contradiction.
\end{proof}

\begin{lemma}
\label{theorem:use-pro-l} In the same setting as in
Lemma~\ref{theorem:only-cdv-points}, if $\dim\Lambda\leqslant 3$
and $\Lambda\cap X'\subset X'\setminus{\text{Sing}(X')}$ is a
finite set, then $\dim\Lambda = 3$ and $\Lambda$ is a tangent
space to $X'$ at a smooth point.
\end{lemma}

\begin{proof}
Let $S'$ be a general hyperplane section of $X'$ passing through
$\Lambda$. Then, if either $\dim\Lambda\leqslant 2$ or $\Lambda$
is not a tangent space, the surface $S'$ will be smooth at
$\Lambda\cap X'$ by Bertini theorem. This gives a birational map
$$
\chi := \pi\big\vert_{S'} : S' \dashrightarrow S
$$
onto a generic surface $S\in |-K_{X}|$ with at worst Du Val
singularities. The map $\chi$ has possible indeterminacies only at
the locus $\Lambda\cap X'$. This implies that $S'$ is a (partial)
minimal resolution of $S$ and $\chi$ is a morphism. Indeed, since
$S'$ is smooth at the \emph{finite} set $\Lambda\cap X'$, by
resolving the indeterminacies of $\chi$ one can see that near
$\Lambda\cap X'$ the surface $S'$ is isomorphic to its minimal
resolution, which must be the minimal resolution for $S$ as well
(note that $\kappa(S) = \kappa(S') = 0$ for birational $\chi$).
This also makes $\chi$ factor through the minimal resolution of
$S$.

Further, $\pi$ is given by the linear subsystem $\mathcal{L}
\subset |-K_{X'}|$ of all hyperplane sections of $X'$ passing
through $\Lambda$, and $\chi$ in turn is defined via
$\mathcal{L}_{S'} := \mathcal{L}\big\vert_{S'}$ by construction
(the linear system $\mathcal{L}_{S'}$ does not have fixed
components because $\Lambda\cap X'$ is finite). Then the previous
discussion implies that the set $\Lambda\cap S'$ (= the base locus
of $\mathcal{L}_{S'}$), hence $\Lambda\cap X'$ as well, is empty.
The latter contradicts the last statement of
Lemma~\ref{theorem:bir-proj}.
\end{proof}

There is the following partial inversion of
Lemma~\ref{theorem:only-cdv-points}:

\begin{lemma}
\label{theorem:non-simple-case} Let $\pi: X' \dashrightarrow X$ be
the linear projection from a singular $\mathrm{cDV}$ point $O \in
X'$ (for $X'$ being as earlier). Then $X$ is an anticanonically
embedded Fano threefold such that $(-K_{X})^{3} = (-K_{X'})^{3} -
2$.
\end{lemma}

\begin{proof}
Consider the blowup $\sigma: W \longrightarrow X'$ of $X'$ at $O$:
$$
\xymatrix{
&&W\ar@{->}[ld]_{\sigma}\ar@{->}[rd]^{\tau}&&\\%
&X'\ar@{-->}[rr]_{\pi}&&X.&}
$$
Projection $\pi$ is given by the linear subsystem $\mathcal{H}
\subset |-K_{X'}|$ of all hyperplane sections of $X'$ passing
through $O$.

Further, since $O \in X'$ is a \emph{singular} $\mathrm{cDV}$
point, the singularity germ $(O \in X')$ is isomorphic to a
hypersurface one
$$
\big(o \in \big(f\big(x,y,z\big) + tg\big(x,y,z,t\big) =
0\big)\big) \subset\mathbb{C}^4,
$$
where $f(x,y,z) = 0$ is an equation of Du Val singularity,
$o\in\mathbb{C}^4$ is the origin, and
$g(x,y,z,t)\in\mathbb{C}[x,y,z,t]$ is some polynomial with $g(o) =
0$. This immediately shows that near $E_{\sigma} :=
\sigma^{-1}(O)$ the threefold $W$ also has only normal
hypersurface singularities.\footnote{Probably, if being extra
careful, one may have to replace $W$ by its normalization in any
of the singular neighborhoods. Such a replacement does not affect
however the forthcoming properties of $\tau$ and $K_W$.} In
particular, $W$ has (at worst) Gorenstein singularities, which are
also canonical (because such are singularities of $X'$).
Furthermore, since locally on the chart $(t\ne 0)$ we have an
isomorphism
$$
W \simeq \big(f\big(xt,yt,zt\big)t^{-2} +
tg\big(xt,yt,zt,t\big)t^{-2} = 0\big),
$$
and similarly for other charts $(x\ne 0)$, etc., we obtain the
equality
$$W = \sigma_*^{-1}(X') = \sigma^*(X') - 2E_{\sigma}$$
of cycles on the blowup $\widetilde{\mathbb{C}^4}$ of
$\mathbb{C}^4$ at $O$. Then, since $K_{\widetilde{\mathbb{C}^4}} =
\sigma^*(K_{\mathbb{C}^4}) + 3E_{\sigma}$, from adjunction applied
to $(\widetilde{\mathbb{C}^4},W)$ and $(\mathbb{C}^4,X')$ we
deduce that
$$
K_{W}=\sigma^{*}(K_{X'}) + E_{\sigma}.
$$
Furthermore, we have $\sigma_{*}^{-1}(H) = \sigma^{*}(H) -
E_{\sigma}$, where $H \in \mathcal{H}$ is a generic surface.

Thus the morphism $\tau: W \longrightarrow X$ is given by the
linear subsystem $\sigma_{*}^{-1}(\mathcal{H}) \subseteq
|-K_{W}|$. In particular, $W$ is a weak Fano threefold, for
$\sigma_{*}^{-1}(\mathcal{H})$ is basepoint-free and $(-K_{W})^{3}
= (-K_{X'})^{3} - 2
> 0$.\footnote{Recall that $X'$ was assumed to be anticanonically embedded.
So, once $(-K_{X'})^{3} \leqslant 2$, it must be a quadric (due to
the $\deg X' \geqslant \text{codim}\ X' + 1$ inequality). But the
latter is clearly impossible for then $(-K_{X'})^{3} = 54$.} Then
$\dim X = 3$ and $\pi$ is birational (see
Lemma~\ref{theorem:bir-proj}). Thus, since $\tau$ is a crepant
morphism, threefold $X$ has at worst canonical Gorenstein
singularities (apply Remark~\ref{remark:K-trivial-contraction-1}
to $W$ in place of $Y$). Moreover, we have $-K_X \simeq
\mathcal{O}_{X}(1)$ by construction, which implies that $X$ is an
anticanonically embedded Fano threefold with $(-K_{X})^{3} =
(-K_{X'})^{3} - 2$.
\end{proof}

\begin{example}
\label{example:examp-5} Let us apply
Lemma~\ref{theorem:non-simple-case} to the threefold $X' :=
\mathbb{P}(6,4,1,1)$ and a point $O\ne P,Q$ (cf.
Example~\ref{example:examp-2} and
Corollary~\ref{theorem:embed-as-intersection-of-quadrics}). Then
we get $X = X_{70}$ from Theorem~\ref{theorem:main-0}. In the
notation of Example~\ref{example:examp-2} and the proof of
Lemma~\ref{theorem:non-simple-case}, we have $\mathrm{Sing}(W) =
\sigma_{*}^{-1}(L)$ and singularities of $W$ are exactly of the
same kind as that of $X'$, i.e., locally near every point from
$\mathrm{Sing}(W)$ threefold $W$ is analytically isomorphic to
$X'$. Then, resolving the singularities of $W$ in the same way as
for $X'$ in Example~\ref{example:examp-2}, we arrive at some
birational morphism $\mu : Y \longrightarrow W$ with $Y$ being
smooth and $K_Y = \mu^{*}(K_W)$. The morphism $f := \mu \circ \tau
: Y \longrightarrow X$ is a terminal $\mathbb{Q}$-factorial
modification of $X$. Note also that $X = X_{70} \subset
\mathbb{P}^{37}$ is an intersection of quadrics by
Corollary~\ref{theorem:embed-as-intersection-of-quadrics}.
\end{example}

\begin{remark}
\label{remark:unique-terminal-modification-5} Let $f_{70}: Y_{70}
\longrightarrow X_{70}$ be a terminal $\mathbb{Q}$-factorial
modification of $X_{70}$. Then it follows from
Example~\ref{example:examp-5} and
Remark~\ref{remark:terminal-modifications-are-connected-by-flops}
that $Y_{70}$ is smooth. Moreover, Example~\ref{example:examp-5}
and
Remarks~\ref{remark:terminal-modifications-are-connected-by-flops},
\ref{remark:unique-terminal-modification-2} show that exceptional
locus of $f_{70}$ is of pure codimension $1$.
\end{remark}

\begin{example}[cf. {\cite[Section 5]{karz}}]
\label{example:examp-6} Let $W :=
\mathbb{P}_{\mathbb{P}^{1}}(\mathcal{O}_{\mathbb{P}^{1}}(5)\oplus\mathcal{O}_{\mathbb{P}^{1}}(2)\oplus\mathcal{O}_{\mathbb{P}^{1}})$
be the rational scroll with tautological divisor $M$ and a fiber
$F$ of the natural projection $W \to \mathbb{P}^{1}$. We have
$-K_{W} \sim 3M - 5F$ and the linear system $|-K_{W}|$ does not
have fixed components. Let us denote by $X$ the image of $W$ under
the rational map $\Phi_{\scriptscriptstyle|-K_{W}|}$. The base
$\mathrm{Bs}(|-K_{W}|)$ is a curve $C\simeq\mathbb{P}^1$ such that
$K_{W} \cdot C = 5$. One can resolve the indeterminacies of
$\Phi_{\scriptscriptstyle|-K_{W}|}$ by three subsequent toric
blowups (starting with $C$). This results in a threefold $Y$ with
basepoint-free linear system $|-K_Y|$ (which is the proper
transform of $|-K_W|$) and $(-K_{Y})^{3} = 66$. In particular, $Y$
is a weak Fano threefold, $X$ is a Fano threefold and $f: =
\Phi_{\scriptscriptstyle|-K_{Y}|} : Y \longrightarrow X$ is its
terminal $\mathbb{Q}$-factorial modification (cf.
Remarks~\ref{remark:K-trivial-contraction-1},
\ref{remark:K-trivial-contraction-2}). Note also that $X$
coincides with the threefold $X_{66}$ from
Theorem~\ref{theorem:main-0}.
\end{example}

\begin{remark}
\label{remark:unique-terminal-modification-6} Let $f_{66}: Y_{66}
\longrightarrow X_{66}$ be a terminal $\mathbb{Q}$-factorial
modification of $X_{66}$. Then it follows from
Example~\ref{example:examp-6} and
Remark~\ref{remark:terminal-modifications-are-connected-by-flops}
that $Y_{66}$ is smooth.
\end{remark}

\bigskip

\section{Some auxiliary results about extremal rays}
\label{section:auxiliary}

Let $X$, $Y$ and $f$ be as in
Proposition~\ref{theorem:terminal-modification}. Consider a
$K_{Y}$-negative extremal contraction $\mathrm{ext}: Y
\longrightarrow Y'$ onto some variety $Y'$ (cf.
Proposition~\ref{theorem:extremal-rays-cone} and
Remark~\ref{remark:K-trivial-contraction-2}). Let us collect
several properties of $Y$, $Y'$ and $\mathrm{ext}$:

\begin{proposition}
\label{theorem:namikawa-smoothing} Let $\dim Y' = 0$, i.e., $X =
Y$. Then $(-K_{X})^{3} \leqslant 64$. Moreover, if $(-K_{X})^{3} =
64$, then $X = \mathbb{P}^{3}$.
\end{proposition}

\begin{proof}
Under the stated assumption, $X$ has at worst terminal factorial
singularities. Then the main result of \cite{Namikawa-smoothing}
implies that there exists a flat morphism $\iota: \mathcal{X}
\longrightarrow \Delta$, where $\Delta := \{t \in \mathbb{C} : |t|
< 1\}$ is the unit disk, such that $\iota^{-1}(0) \simeq X$ and
$X_{t} := \iota^{-1}(t)$ is a smooth Fano threefold for all $0 \ne
t \in \Delta$. In particular, we get $(-K_{X})^{3} \leqslant 64$
(see \cite{VA-1}, \cite{VA-2}).

Further, it follows from the proof of Theorem 1.4 in
\cite{jahnke-rad-smooth} that for every $t \in \Delta$ there is a
group isomorphism $\varphi: \mathrm{Pic}(X) \simeq
\mathrm{Pic}(X_{t})$ such that $\varphi(K_{X}) = K_{X_{t}}$. Let
us now assume that $(-K_{X})^{3} = 64$. Then $(-K_{X_{t}})^{3} =
64$ as well, and \cite{VA-1}, \cite{VA-2} imply that $X_{t} \simeq
\mathbb{P}^{3}$ for all $0 \ne t \in \Delta$. Moreover, since
$\mathrm{Pic}(X) \simeq \mathrm{Pic}(X_{t}) \simeq \mathbb{Z}$ for
all $t \in \Delta$, the isomorphism $\varphi$ gives the equalities
$r(X) = r(X_{t}) = 4$ for Fano indices of $X$ and $X_{t}$. Then
from \cite[Theorem 3.1.14]{iskovskikh-prokhorov} we deduce that $X
= \mathbb{P}^{3}$. Proposition~\ref{theorem:namikawa-smoothing} is
completely proved.
\end{proof}

\begin{remark}
\label{remark:namikawa-smoothing-rem} Let $(-K_{X})^{3} = 64$.
Then the same argument as in the proof of
Proposition~\ref{theorem:namikawa-smoothing} shows that
singularities of $X$ are worse than terminal unless $X =
\mathbb{P}^3$. Indeed, the main result of
\cite{Namikawa-smoothing} applies to the terminal Gorenstein $X$,
and \cite{jahnke-rad-smooth} gives that $X = \mathbb{P}^3$ as
earlier. In particular, $f$ is not a small morphism (for otherwise
the singularities of $X$ are obviously terminal), provided that
$(-K_{X})^{3} = 64$ and $X \ne Y$. The same reasoning applies to
show that the estimate $(-K_{X})^{3} \leqslant 64$ holds as well
for all terminal Gorenstein $X$.
\end{remark}

\begin{proposition}[see {\cite{Cutkosky}, \cite{Mori-flip}}]
\label{theorem:class-of-extr-rays} If $\mathrm{ext}$ is a
birational morphism, then the following hold:

\begin{itemize}

\item $\mathrm{ext}$ is divisorial, with an irreducible exceptional divisor $E$;

\item if $P := \mathrm{ext}(E)$ is a point, then $\mathrm{ext}$ is the blowup of $Y'$ at
$P$. Furthermore, $Y'$ is a weak Fano threefold with at worst
terminal singularities and such that $(-K_{Y'})^{3} >
(-K_{Y})^{3}$. Moreover, $Y'$ is factorial except for the case
when $E \simeq \mathbb{P}^{2}$, $\mathcal{O}_{E}(E) \simeq
\mathcal{O}_{\mathbb{P}^{2}}(2)$ and $(O \in Y')$ is the
singularity of type $\displaystyle\frac{1}{2}(1,1,1)$;

\item if $C : = \mathrm{ext}(E)$ is a curve, then $Y'$ is
smooth along $C$, the scheme $C$ is reduced and irreducible, and
$\mathrm{ext}$ is the blowup of $Y'$ at $C$.

\end{itemize}

\end{proposition}

\begin{corollary}
\label{theorem:0-cor-contraction} In the assumptions of
Proposition~\ref{theorem:class-of-extr-rays}, if $(\mathrm{ext}(E)
\in Y')$ is the singularity of type
$\displaystyle\frac{1}{2}(1,1,1)$, then $f(E)$ is a plane on $X$.
\end{corollary}

\begin{proof}
Indeed, since $E \simeq \mathbb{P}^{2}$ and $\mathcal{O}_{E}(E)
\simeq \mathcal{O}_{\mathbb{P}^{2}}(2)$, we have
$$
K_{Y}^{2} \cdot E = \big(\mathrm{ext}^{*}\big(K_{Y'}\big) +
\frac{1}{2}E\big)^{2} \cdot E = \frac{1}{4}E^{3} = 1,
$$
which implies that $f(E)$ is a plane.
\end{proof}

\begin{proposition}[see {\cite[Proposition-definition 4.5]{Prokhorov-degree}}]
\label{theorem:1-contraction} If $C := \mathrm{ext}(E)$ is a
curve, then $Y'$ has at worst terminal factorial singularities and
one of the following holds:

\begin{itemize}

\item $Y'$ is a weak Fano threefold with $(-K_{Y'})^{3} \geqslant
(-K_{Y})^{3}$;

\smallskip

\item $K_{Y'} \cdot C > 0$ and $C$ is the only curve on $Y'$ which intersects $K_{Y'}$ positively.
Moreover, $C \simeq \mathbb{P}^{1}$ and either $E \simeq
\mathbb{P}^{1} \times \mathbb{P}^{1}$ or $E \simeq
\mathbb{F}_{1}$.

\end{itemize}
\end{proposition}

\begin{corollary}[see {\cite[Corollary 4.9]{Prokhorov-degree}}]
\label{theorem:1-cor-contraction} In the assumptions of
Proposition~\ref{theorem:1-contraction}, the following hold:

\begin{itemize}

\item if $E \simeq \mathbb{P}^{1} \times \mathbb{P}^{1}$, then $f(E)$ is a line on $X$ and $X$ is singular along $f(E)$;

\smallskip

\item if $E \simeq \mathbb{F}_{1}$, then $f(E)$ is a plane on $X$.

\end{itemize}

\end{corollary}

\begin{lemma}[{cf. \cite[Lemmas 3.10, 3.11]{karz}}]
\label{theorem:singularities-of-X} For $X := \mathbb{P}(6,4,1,1)$,
the locus $L := \mathrm{Sing}(X)$ is the unique line on $X$.
\end{lemma}

\begin{proof}
Notice first that $L\simeq\mathbb{P}^1$ and $-K_{X} \cdot L = 1$.
Hence, as $-K_X\sim\mathcal{O}_X(1)$, the curve $L$ is a line on
$X$.

Let $L_{0} \ne L$ be another line on $X$. Then $L_0$ intersects
$L$ in at most one point. It follows from
Example~\ref{example:examp-2} that $-K_X = 12D$ for some Weil
divisor $D$ with $L_0\not\subset D$. Furthermore, $D$ is either
$6$-, $4$-, $2$- or just Cartier near $L_0\cap L$, hence near
$L_0$. In particular, we get $D\cdot L_0\geqslant 1/6$, which
contradicts $$D\cdot L_0 = \frac{1}{12}(-K_X)\cdot L_0 =
\frac{1}{12}.$$
\end{proof}

\begin{lemma}
\label{theorem:singularities-of-70} The locus
$\mathrm{Sing}(X_{70})$ consists of a unique point $o$ and
singularity $(o \in X_{70})$ is worse than $\mathrm{cDV}$.
\end{lemma}

\begin{proof}
The fact that $\mathrm{Sing}(X_{70})$ consists of a unique point
follows from Lemma~\ref{theorem:singularities-of-X} and the
construction of $X_{70}$ in Example~\ref{example:examp-5}.
Further, if $(o \in X_{70})$ is a $\mathrm{cDV}$ singularity, then
$X_{70}$ has at worst terminal Gorenstein singularities (see
\cite[Theorem 5.34]{Kollar-Mori}). In this case, we derive a
contradiction from Remark~\ref{remark:namikawa-smoothing-rem}.
\end{proof}

\begin{remark}
\label{remark:rem-singularities-of-70} Exactly the same argument
as in the proof of Lemma~\ref{theorem:singularities-of-70} shows
that the only singularity on $\mathbb{P}(3,1,1,1)$ (see
Example~\ref{example:examp-1}) is also worse than $\mathrm{cDV}$.
\end{remark}

Lemmas~\ref{theorem:singularities-of-X} and
\ref{theorem:singularities-of-70} lead to some fruitful
information about the geometry of the threefold $X_{70}$:

\begin{proposition}
\label{theorem:contraction-to-curve-C} Every line on $X := X_{70}$
passes through the singular point $o \in X$.
\end{proposition}

\begin{proof}
Let us use the notation from the proof of
Lemma~\ref{theorem:non-simple-case}. Suppose that there is a line
$\Gamma \subset X$ such that $o\not\in\Gamma$.

The blowup of $X$ at $\Gamma$ resolves indeterminacies of the
linear projection from $\Gamma$ and leads to a weak Fano threefold
$V$ together with birational morphism $V \longrightarrow X_{66}
\subset \mathbb{P}^{35}$ onto a Fano threefold $X_{66}$ such that
$-K_{X_{66}} \simeq \mathcal{O}_{X_{66}}(1)$ and $(-K_{X_{66}})^3
= 66$ (cf. Lemma~\ref{theorem:bir-proj} and \eqref{deg-estimate}).

Recall that there is a birational map $\pi : X' :=
\mathbb{P}(6,4,1,1) \dashrightarrow X$.

\begin{lemma}
\label{theorem:prop-trans-conic} $\Gamma \not\subset
\tau(E_{\sigma})$.
\end{lemma}

\begin{proof}
By construction, the surface $S := \tau(E_{\sigma}) \subset X
\subset \mathbb{P}^{37}$ is a quadratic cone, i.e., $S$ is a
minimal surface of degree $2$, singular at $o =
\tau(\sigma_{*}^{-1}(L))$, where $L := \mathrm{Sing}(X')$. Thus,
since $o\not\in\Gamma$, we obtain that $\Gamma \not\subset
\tau(E_{\sigma})$.
\end{proof}

Lemma~\ref{theorem:prop-trans-conic} implies that we can consider
the curve $C := \pi_{*}^{-1}(\Gamma) =
\sigma(\tau_{*}^{-1}(\Gamma))$ on $X'$.

\begin{lemma}
\label{theorem:prop-trans-conic-1} $C$ is a smooth conic.
\end{lemma}

\begin{proof}
Recall that $\pi$ is the linear projection from a point on $X'$.
Then $C$ is either a line or a conic.
Lemma~\ref{theorem:singularities-of-X} shows that $C$ is a conic.
\end{proof}

Lemma~\ref{theorem:prop-trans-conic} yields the linear projection
$p : X' \dashrightarrow X_{66}$ from a plane $\Pi \subset
\mathbb{P}^{38} \supset X'$ such that $\Pi \cap X' = C$. Then we
get a commutative diagram
\begin{equation}
\nonumber \xymatrix{
Y_{66}\ar@{->}[d]_{f_{66}}\ar@{->}[r]^{\rho}&Y'\ar@{->}[d]^{f'}\\
X_{66}&\ar@{-->}[l]_{p}X',}
\end{equation}
where $f' : Y' \longrightarrow X'$ and $f_{66} : Y_{66}
\longrightarrow X_{66}$ are terminal $\mathbb{Q}$-factorial
modifications of $X'$ and $X_{66}$, respectively, and $\rho$ is
the blowup of $Y'$ at the curve $f_{*}^{'-1}(C)$. In particular,
we have
\begin{equation}
\label{can-under-blow-111} K_{Y_{66}} = \rho^{*}(K_{Y'}) + E',
\end{equation}
where $E'$ is the $\rho$-exceptional divisor. On the other hand,
since $-K_{X'} \cdot C = 2$ (see
Lemma~\ref{theorem:prop-trans-conic-1}), we get
\begin{equation}
\nonumber \mathcal{O}_{X'}(1) \cdot C = \frac{1}{6}
\end{equation}
(cf. the proof of Lemma~\ref{theorem:singularities-of-X}), which
implies that $L \cap C \ne \emptyset$. In particular, there is a
curve $Z \subset Y_{66}$ such that $K_{Y'} \cdot \rho_{*}(Z) = 0$
and $E' \cdot Z
> 0$ (cf. Remark~\ref{remark:unique-terminal-modification-2} and the construction of $Y'$). Then from \eqref{can-under-blow-111} we get $K_{Y_{66}}
\cdot Z > 0$, which is impossible for $-K_{Y_{66}}$ nef.

Proposition~\ref{theorem:contraction-to-curve-C} is completely
proved.
\end{proof}

\bigskip

\section{Proof of Theorem~\ref{theorem:main-1}: the Mori fiber space case}
\label{section:mori-fiber-space}

\renewcommand{\labelenumi}{(\theenumi)}
\renewcommand{\theenumi}{\Alph{enumi}}

Let $X$, $Y$ and $f$ be as in
Proposition~\ref{theorem:terminal-modification}. Through the rest
of the present Section we consider Fano threefold $X$ such that
$(-K_{X}^{3}) = 64$ and $\mathrm{ext}: Y \longrightarrow Y'$ is a
Mori fiber space contraction (recall that $\mathrm{ext}: Y
\longrightarrow Y'$ is a $K_{Y}$-negative extremal contraction
onto some variety $Y'$). We will also assume that $\dim Y' > 0$
(cf. Proposition~\ref{theorem:namikawa-smoothing}).

\begin{lemma}[{see \cite[Proposition
4.11]{Prokhorov-degree}}] \label{theorem:dim-y-prime-is-2} $\dim
Y' \ne 1$.
\end{lemma}

\begin{lemma}
\label{theorem:when-dim-2-x-is-singular} Let $\dim Y' = 2$. Then
$X \ne \mathbb{P}^3$. In particular, $f$ is not a small morphism.
\end{lemma}

\begin{proof}
Suppose that $X = \mathbb{P}^3$. Then we get $\mathbb{P}^3 = X =
Y$ by definition. But the latter is impossible because $\dim Y' =
2$ implies $\rho(Y) > 1$. This and
Remark~\ref{remark:namikawa-smoothing-rem} finish the proof.
\end{proof}

\begin{lemma}[{see \cite[Proposition
5.2]{Prokhorov-degree}}] \label{theorem:every-thing-is-good} Let
$\dim Y' = 2$. Then the following hold:

\begin{enumerate}

\item\label{A-case} $\mathrm{ext}: Y
\longrightarrow Y'$ is a $\mathbb{P}^{1}$-bundle;

\item\label{B-case} the surface $Y'$ is smooth;

\item\label{C-case} $-K_{Y'}$ is nef and big.

\end{enumerate}

\end{lemma}

\begin{proposition}
\label{theorem:no-1-curves-on-y} Let $\dim Y' = 2$ and the surface
$Y'$ does not contain $(-1)$-curves. Then $X$ is isomorphic to the
cone from Example~\ref{example:examp-0} (see $(\ref{num-0})$ in
Theorem~\ref{theorem:main-1}).
\end{proposition}

\begin{proof}
It follows from Lemma~\ref{theorem:every-thing-is-good} that $Y =
\mathbb{P}(\mathcal{E})$ for some rank $2$ vector bundle
$\mathcal{E}$ on $Y'$ so that $\mathrm{ext}: Y \longrightarrow Y'$
coincides with the natural projection $\mathbb{P}(\mathcal{E}) \to
Y'$. Set $c_{i} := c_{i}(\mathcal{E})$ for $i \in\{1,2\}$. Let
also $D$ be the tautological section of the
$\mathbb{P}^{1}$-bundle $\mathrm{ext}: Y \longrightarrow Y'$. Then
from the relative exact Euler sequence (see \cite[Proposition
4.26]{Mori-Mukai}) we get the equality
\begin{equation}
\label{can-class-formula-on-w-1} -K_{Y} = 2D +
\mathrm{ext}^{*}\big(-K_{Y'} - c_{1}\big).
\end{equation}
On the other hand, from the Hirsch formula (see
\cite{grothendieck-chern-classes}) we get the identities
\begin{equation}
\label{hirsch-formulae-1} D^2 \equiv D \cdot
\mathrm{ext}^{*}(c_{1}) - \mathrm{ext}^{*}(c_{2})
\qquad\mbox{and}\qquad D^3 = c_{1}^{2} - c_{2}.
\end{equation}
From \eqref{can-class-formula-on-w-1} and
\eqref{hirsch-formulae-1} we obtain the following formula:
\begin{equation}
\label{equation:formula-for-degree-1}
\begin{array}{c}
\big(-K_{Y}\big)^{3} = 6K_{Y'}^{2} + 2c_{1}^{2} - 8c_{2}.
\end{array}
\end{equation}
We also have
\begin{equation}
\label{equation:Rieman-Roch-formula-1}
\begin{array}{c}
\chi(Y', \mathcal{E}) =
\displaystyle\frac{1}{2}\big(c_{1}^{2}-2c_{2}-K_{Y'} \cdot
c_{1}\big)+2
\end{array}
\end{equation}
by the Riemann-Roch formula on $Y'$ (see \cite{hirtz}).

Further, from Lemma~\ref{theorem:every-thing-is-good} we obtain
that $Y'$ is either $\mathbb{P}^{2}$, $\mathbb{P}^{1} \times
\mathbb{P}^{1}$ or $\mathbb{F}_{2}$, since $Y'$ does not contain
$(-1)$-curves by assumption. Then we get the following:

\begin{lemma}
\label{theorem:restriction-on-movable-curve} Let $Z \subset Y'$ be
an irreducible rational curve such that $\dim |Z|
> 0$ and $\mathcal{E}\big\vert_{Z} \simeq
\mathcal{O}_{\mathbb{P}^{1}}(d_{1}) \oplus
\mathcal{O}_{\mathbb{P}^{1}}(d_{2})$ for some $d_{1}, d_{2} \in
\mathbb{Z}$. Then $|d_{1} - d_{2}| \leqslant 2 + (Z^{2})$.
\end{lemma}

\begin{proof}
This follows from exactly the same arguments as in the proof of
\cite[Lemma 10.6]{Prokhorov-degree} (recall that the linear system
$|-nK_{Y}|$ is basepoint-free for $n$ large).
\end{proof}

\renewcommand{\labelenumi}{{\bf(\theenumi)}}
\renewcommand{\theenumi}{\arabic{enumi}}

\begin{lemma}
\label{theorem:y-prime-is-p-2} $Y' \ne \mathbb{P}^{2}$.
\end{lemma}

\begin{proof}
Suppose that $Y' = \mathbb{P}^{2}$. Then we get $H^{2i}(Y',
\mathbb{Z}) \simeq \mathbb{Z}$ for $i \in\{1,2\}$. Hence we may
assume that both $c_i$ are integers. We have two cases:

\begin{enumerate}

\item $c_{1}$ is even. Then we may assume that $c_{1} = 0$ up to
twisting $\mathcal{E}$ by a line bundle (cf. \cite{hirtz}). In
this case, from \eqref{equation:formula-for-degree-1} we obtain
the equality
$$
64 = \big(-K_{Y}\big)^{3} = 54 - 8c_{2},
$$
which implies that $c_{2} = -5/4 \not\in \mathbb{Z}$, a
contradiction.

\item $c_{1}$ is odd. Then, as above, we may assume that $c_{1} = -K_{Y'}$. In
this case, from \eqref{can-class-formula-on-w-1} we get the
equality
$$
-K_{Y} = 2D.
$$
Identify $D$ with a generic element in $|D|$. This is a weak del
Pezzo surface with at worst Du Val singularities (see
\cite{Shin}). Then from the adjunction formula we obtain
$$
-K_{D} = D\big\vert_{D},
$$
and since $(-K_{Y})^{3} = 64$, we get $K_{D}^{2} = 8$. Now, since
$Y'\simeq\mathbb{P}^2$, one easily gets that either
$D\simeq\mathbb{P}^2$ or $\mathbb{F}_1$ (see e.g. \cite[Theorem
8.1.5]{dolgachev}). But the former case is impossible for
$K_{D}^{2} = 8$. Hence we have $D \simeq \mathbb{F}_1$. Let $E_f$
be the exceptional locus of the contraction $f: Y \longrightarrow
X$. Suppose that $Z := E_f\cap D\ne \emptyset$. Notice that $\dim
Z > 0$ because $E_f$ is covered by $K_Y$-trivial curves. The same
reasoning gives either $E_f \subset D$ or $K_Y \cdot Z = 0$. But
then either $\mathrm{ext}_*(D) = 0$ or $Z$ is a fiber of
$\mathrm{ext}$, respectively, thus all being absurd. Hence we get
$Z = \emptyset$. On the other hand, $\dim E_f = 2$ (see
Lemma~\ref{theorem:when-dim-2-x-is-singular}) and $D =
\mathbb{F}_1$ contains a fiber of $\mathrm{ext}$ by construction,
which implies that $D \cap E_f\ne \emptyset$, a contradiction.
\end{enumerate}
\end{proof}

\begin{lemma}
\label{theorem:y-prime-is-p-1-times-p-1} Let $Y' = \mathbb{P}^{1}
\times \mathbb{P}^{1}$. Then $X$ is isomorphic to the cone from
Example~\ref{example:examp-0}.
\end{lemma}

\begin{proof}
We have
$$
H^{4}(Y', \mathbb{Z}) \simeq \mathbb{Z}\qquad\mbox{and}\qquad
H^{2}(Y', \mathbb{Z}) \simeq \mathbb{Z}\cdot h \oplus
\mathbb{Z}\cdot l.
$$
Hence $c_{1} := ah + bl$ and $c_{2} := c$ for some $a, b, c \in
\mathbb{Z}$. We get two cases:

\begin{enumerate}

\item $a$ and $b$ are both even. Then we may assume that $c_{1} = -K_{Y'}$ up to
twisting $\mathcal{E}$ by a line bundle (cf. \cite{hirtz}). In
this case, from \eqref{can-class-formula-on-w-1} we get the
equality
\begin{equation}
\label{k-y-on-y} -K_{Y} = 2D.
\end{equation}
Then from the adjunction formula we obtain
\begin{equation}
\label{k-d-e} -K_{D} = D\big\vert_{D}.
\end{equation}
Recall that $f: Y \longrightarrow X$ contracts a divisor (see
Lemma~\ref{theorem:when-dim-2-x-is-singular}). Furthermore, since
$D \simeq \mathbb{P}^{1} \times \mathbb{P}^{1}$, \eqref{k-y-on-y}
and \eqref{k-d-e} imply that either $E_{f} \subset D$ or $E_{f}
\cap D = \emptyset$ for the $f$-exceptional locus $E_{f}$. Notice
also that the restriction
$$
H^{0}(Y, \mathcal{O}_{Y}(D)) \to
H^{0}(D, \mathcal{O}_{D}(D))
$$
is surjective because $H^{1}(Y, \mathcal{O}_{Y}) = 0$ by
Kawamata-Viehweg vanishing theorem. Hence we get a morphism
$\Phi_{|D|} : Y \longrightarrow \mathbb{P}^9$ such that
$\Phi_{|D|}(Y) \simeq X$ is the cone over the anticanonically
embedded surface $\mathbb{P}^{1} \times \mathbb{P}^{1}$ with the
vertex $\Phi_{|D|}(E_{f})$.

\item $a$ or $b$ is odd. Then exactly as in the proof of \cite[Proposition 5.2]{Prokhorov-degree} we get contradiction in this case (see
\cite[5.11, {\bf Case} $Z \simeq \mathbb{P}^{1} \times
\mathbb{P}^{1}$]{Prokhorov-degree}).

\end{enumerate}
\end{proof}

Let us now show that $Y' \ne \mathbb{F}_{2}$. Suppose that $Y' =
\mathbb{F}_{2}$. Then, as in the proof of
Lemma~\ref{theorem:y-prime-is-p-1-times-p-1}, we get $c_{1} := ah
+ bl$ and $c_{2} := c$ for some $a, b, c \in
\mathbb{Z}$.\footnote{The arguments below are similar to those in
the proof of \cite[ Theorem 1.5]{Prokhorov-degree} (see \cite[{\bf
10.19}]{Prokhorov-degree}). We partly reproduce them here for
convenience of the future reference.}

\begin{lemma}
\label{theorem:c-2-is-minus-2} Up to twisting $\mathcal{E}$ by a
line bundle, $c_{1} = -2h - 2l$ and $c = -2$.
\end{lemma}

\begin{proof}
We have two cases:

\begin{enumerate}

\item $a$ and $b$ are both even. Then, as above, we may assume that $c_{1} = -2h - 2l$. In
this case, from \eqref{equation:formula-for-degree-1} we get the
equality
$$
64 = \big(-K_{Y}\big)^{3} = 48 - 8c,
$$
which implies that $c = -2$.

\item $a$ or $b$ is odd. Then exactly as in the proof of \cite[Proposition 5.2]{Prokhorov-degree} we get that $c = -2$ in this case (see
\cite[5.11, {\bf Case} $Z \simeq
\mathbb{F}_{2}$]{Prokhorov-degree}).

\end{enumerate}
\end{proof}

\begin{lemma}
\label{theorem:h-0-is-non-zero} $H^{0}(Y',\mathcal{E}) \ne 0$.
\end{lemma}

\begin{proof}
From \eqref{equation:Rieman-Roch-formula-1}, Serre duality and
Lemma~\ref{theorem:c-2-is-minus-2} we obtain
$$
h^{0}\big(Y',\mathcal{E}\big) + h^{0}\big(Y',
\mathcal{E}\otimes\det\mathcal{E}^{*}\otimes\omega_{Y'}\big)
\geqslant \frac{1}{2}\big(\big(2h+2l\big)^{2} - 2c + 2K_{Y'} \cdot
\big(h + l\big)\big) + 2 = 2.
$$
Then, since $\det\mathcal{E}^{*}\otimes\omega_{Y'} \simeq
\mathcal{O}_{Y'}(-2l)$ for $\det\mathcal{E}^{*} = 2h + 2l$ (see
Lemma~\ref{theorem:c-2-is-minus-2}) and $\omega_{Y'} =
\mathcal{O}_{\mathbb{F}_2}(-2h - 4l)$, we get
$H^{0}(Y',\mathcal{E}) \ne 0$.
\end{proof}

Lemma~\ref{theorem:h-0-is-non-zero} implies that there exists a
non-zero section $s \in H^{0}(Y', \mathcal{E})$. Let $Z \subset
Y'$ be the zero locus of $s$. Then, since $c = -2$ by
Lemma~\ref{theorem:c-2-is-minus-2}, we have $\dim Z = 1$ (see
\cite{grothendieck-chern-classes}), and hence $Z \sim q_{1}h +
q_{2}l$ for some integers $q_i \geqslant 0$. Simple properties of
Chern classes (see \cite{hirtz}) imply that
$$
\mathcal{E}\big\vert_{L} =
\mathcal{O}_{\mathbb{P}^{1}}(q_{1})\oplus
\mathcal{O}_{\mathbb{P}^{1}}(-2-q_{1})
$$
for generic curve $L \in |l|$. Then
Lemma~\ref{theorem:restriction-on-movable-curve} immediately gives
$2q_{1} + 2 \leqslant 2$.

Thus we get $q_{1} = 0$ and $Z$ is contained in the fibres of the
$\mathbb{P}^{1}$-bundle $p_2 : Y' = \mathbb{F}_{2} \to
\mathbb{P}^{1}$ (hence $q_2 > 0$). On the other hand, we have
\begin{equation}
\label{for-e} \mathcal{E}\big\vert_{L'} =
\mathcal{O}_{\mathbb{P}^{1}}(q_{2})\oplus
\mathcal{O}_{\mathbb{P}^{1}}(-2-q_{2})
\end{equation}
for generic curve $L' \in |h + 2l|$. Choose a point $o \in L'$ and
a local coordinate $t$ in a neighborhood $o \in U \subset Y'$ of
$o$ such that the equation $t = t_{0}$, for a fixed $t_{0} \in
\mathbb{C}$, determines the fiber of the morphism
$p_{2}\big\vert_{U} : U  \to L' = \mathbb{P}^{1}$. Then, since the
locus $Z$ is contained in the fibres of $p_2$, we obtain that the
section $s\big\vert_{L'}$ of $\mathcal{E}\big\vert_{L'}$ on $U
\cap L'$ is determined by the set of pairs $(f_{1}(t), f_{2}(t))$
for some holomorphic functions $f_i$, satisfying $f_i(o) = 0$ for
$o := Z \cdot L'$, where $t$ varies on $U \cap L'$. Note that we
may assume $t$ to be the patching function for $L' =
\mathbb{P}^{1}$. Then it follows from \eqref{for-e} that one must
have $2q_2 + 2 = 0$. The latter contradicts $q_2
> 0$.

Proposition~\ref{theorem:no-1-curves-on-y} is completely proved.
\end{proof}

\begin{remark}
\label{theorem:no-1-curves-on-y-rem} Let $\mathcal{E}$ be a rank
$2$ vector bundle on a minimal surface $S := \mathbb{F}_n$ for
some $n \geqslant 0$. Then from the proof of
Proposition~\ref{theorem:no-1-curves-on-y} one derives the
following observation. Namely, if $\chi(S,\mathcal{E}) > 0$ (or
$H^0(S,\mathcal{E}) \ne 0$) and $c_1(\mathcal{E}) = a_1h + a_2l$,
with both $a_i\in\{-1,-2\}$, then the proof of
Lemma~\ref{theorem:h-0-is-non-zero} gives $H^{0}(S,\mathcal{E})
\ne 0$. Furthermore, if the estimate $c_2(\mathcal{E}) < 0$ holds,
then the zero locus of a non-zero section from
$H^{0}(S,\mathcal{E})$ is a curve $\sim q_1h + q_2l$ as earlier.
Now, restricting $\mathcal{E}$ to generic curve from $|l|$ and
applying \cite[Lemma 10.6]{Prokhorov-degree} we obtain that $q_1 =
0$. Finally, restricting $\mathcal{E}$ to generic curve from $|h +
nl|$ we get $2q_2 - a_2 = 0$ exactly as above, which gives
contradiction. Therefore there is no such $\mathcal{E}$ with
prescribed conditions on $\chi(S,\mathcal{E})$ and $c_i$.
\end{remark}

\begin{proposition}
\label{theorem:there-is-1-curve-on-y} Let $\dim Y' = 2$ and the
surface $Y'$ contains a $(-1)$-curve. Then $X$ is isomorphic to
the cone from Example~\ref{example:examp-4} (see $(\ref{num-3})$
in Theorem~\ref{theorem:main-1}).
\end{proposition}

\begin{proof}

\renewcommand{\labelenumi}{{\bf(\theenumi)}}
\renewcommand{\theenumi}{\arabic{enumi}}

Recall that $\mathrm{ext}: Y \longrightarrow Y'$ is a
$\mathbb{P}^{1}$-bundle and $Y'$ is smooth with $-K_{Y'}$ nef and
big (see Lemma~\ref{theorem:every-thing-is-good}). Then we have
$Y' = \mathbb{F}_{1}$,
$$
H^{4}(Y', \mathbb{Z}) \simeq \mathbb{Z},\qquad H^{2}(Y',
\mathbb{Z}) \simeq \mathbb{Z}\cdot h \oplus \mathbb{Z}\cdot l,
$$
and also $Y = \mathbb{P}(\mathcal{E})$ for some rank $2$ vector
bundle $\mathcal{E}$ on $Y'$. Hence $c_{1}(\mathcal{E}) := ah +
bl$ and $c_{2}(\mathcal{E}) := c$ for some  $a, b, c \in
\mathbb{Z}$. We get four cases:

\begin{enumerate}

\item $a$ is odd and
$b$ is even. Then we may assume that $c_{1}(\mathcal{E}) = h$ up
to twisting $\mathcal{E}$ by a line bundle (cf. \cite{hirtz}). In
this case, from the formula similar to
\eqref{equation:formula-for-degree-1} we get the equality
$$
64 = \big(-K_{Y}\big)^{3} = 46 - 8c,
$$
which implies that $c = -9/4 \not\in \mathbb{Z}$, a contradiction.

\item $a$ and $b$ are both odd. Then we may assume that
$c_{1}(\mathcal{E}) = h + l$. In this case, from the formula
similar to \eqref{equation:formula-for-degree-1} we get the
equality
$$
64 = \big(-K_{Y}\big)^{3} = 50 - 8c,
$$
which implies that $c = -7/4 \not\in \mathbb{Z}$, a contradiction.

\item $a$ is even and $b$ is odd. Then we may assume that
$c_{1}(\mathcal{E}) = -K_{Y'}$. In this case, the arguments in the
proof of Lemma~\ref{theorem:y-prime-is-p-1-times-p-1} apply
verbatim to $Y' = \mathbb{F}_{1}$, and so $X$ is isomorphic to the
cone from Example~\ref{example:examp-4}.

\item $a$ and $b$ are both even. Then we may assume that
$c_{1}(\mathcal{E}) = -2h - 2l$. In this case, from the formula
similar to \eqref{equation:formula-for-degree-1} we get the
equality
$$
64 = \big(-K_{Y}\big)^{3} = 56 - 8c,
$$
which implies that $c = -1$. Further, from
\eqref{equation:Rieman-Roch-formula-1} and Serre duality we get
$$
h^{0}\big(Y',\mathcal{E}\big) + h^{0}\big(Y',
\mathcal{E}\otimes\det\mathcal{E}^{*}\otimes\omega_{Y'}\big)
\geqslant \frac{1}{2}\big(\big(2h+2l\big)^{2} - 2c + 2K_{Y'} \cdot
\big(h + l\big)\big) + 2 = 2,
$$
and we can thus apply Remark~\ref{theorem:no-1-curves-on-y-rem} to
derive a contradiction.
\end{enumerate}

Proposition~\ref{theorem:there-is-1-curve-on-y} is completely
proved.
\end{proof}

\bigskip

\section{Proof of Theorem~\ref{theorem:main-1}: the case of contraction to a weak Fano threefold}
\label{section:special-type-contraction}

Let $X$, $Y$ and $f$ be as in
Proposition~\ref{theorem:terminal-modification}. Through the rest
of the present Section we consider Fano threefold $X$ such that
$(-K_{X}^{3}) = 64$ and $\mathrm{ext}: Y \longrightarrow Y'$ is a
birational contraction with exceptional divisor $E$ (cf.
Proposition~\ref{theorem:class-of-extr-rays}). We will also assume
that $Y'$ is a weak Fano threefold with at worst terminal
factorial singularities. This assumption and
Remark~\ref{remark:K-trivial-contraction-1} imply that $Y'$ is a
terminal $\mathbb{Q}$-factorial modification of some Fano
threefold $X'$. Let us denote by $f': Y' \longrightarrow X'$ the
corresponding morphism with exceptional locus $E_{f'}$. Let also
$g'$ be the genus of $X'$.

In what follows, we identify $X$ with its image
$\Phi_{\scriptscriptstyle|-K_{X}|}(X)$ in $\mathbb{P}^{34}$ (cf.
\eqref{deg-estimate}), which is possible by
Corollary~\ref{theorem:embed-as-intersection-of-quadrics}. Recall
that we have $K_{Y} \sim f^{*}(K_{X}),\ -K_{X} \simeq
\mathcal{O}_{X}(1)$, and so the anticanonical linear system
$|-K_{Y}|$ is basepoint-free.

\begin{lemma}
\label{theorem:x-prime-is-ac-embedded} The linear system
$|-K_{X'}|$ gives an embedding of $X'$ into $\mathbb{P}^{g'+1}$ so
that the image $X_{2g'-2} := \Phi_{|-K_{X'}|}(X')$ is an
intersection of quadrics.
\end{lemma}

\begin{proof}
We have $(-K_{Y'})^{3} \geqslant (-K_{Y})^{3} = 64$ (see
Propositions~\ref{theorem:class-of-extr-rays},
\ref{theorem:1-contraction}). Then
Corollary~\ref{theorem:embed-as-intersection-of-quadrics} proves
the claim.
\end{proof}

In what follows, we identify $X'$ with its image
$\Phi_{\scriptscriptstyle|-K_{X'}|}(X')$ in $\mathbb{P}^{g' + 1}$
(cf. \eqref{deg-estimate}), which is possible by
Lemma~\ref{theorem:x-prime-is-ac-embedded}. Consider the diagram
\begin{equation}
\label{comm-diagr} \xymatrix{
Y \ar@{->}[d]_{f}\ar@{->}[r]^{\mathrm{ext}}&Y'\ar@{->}[d]^{f'}\\
X&\ar@{-->}[l]_{p}X'}
\end{equation}
with induced birational map $p$.

For future reference, let us state the following result, similar
to the one that has been used already at the end of the proof of
Proposition~\ref{theorem:contraction-to-curve-C}:

\begin{lemma}
\label{theorem:projection-when-curve-rem} The locus
$\mathrm{ext}(E)$ does not intersect any $K_{Y'}$-trivial curve at
a finite number of points.
\end{lemma}

\begin{proof}
Indeed, otherwise there is a curve $Z \subset Y$ such that $K_{Y'}
\cdot \mathrm{ext}_{*}(Z) = 0$ and $E \cdot Z
> 0$. Recall that $Y'$ is terminal and $\mathrm{ext}$ is the blowup of $\mathrm{ext}(E)$  (see
Proposition~\ref{theorem:class-of-extr-rays}). Then we get $K_{Y}
= \mathrm{ext}^{*}(K_{Y'}) + aE$ for some $a > 0$, and so $K_{Y}
\cdot Z > 0$, which is impossible for $-K_{Y}$ nef.
\end{proof}

\begin{lemma}
\label{theorem:projection-when-curve} Let $C := \mathrm{ext}(E)$
be a curve. Then $p$ is the linear projection from a subspace
which cuts out $f'(C)$ (as a scheme) on $X'$.
\end{lemma}

\begin{proof}
Recall that $C$ is reduced and irreducible, $Y'$ is smooth near
$C$, and $\mathrm{ext}$ is the blowup of $Y'$ at $C$. In
particular, we have
\begin{equation}
\nonumber K_{Y} = \mathrm{ext}^{*}(K_{Y'})+E.
\end{equation}
Thus $f$ is given by the linear system $|-K_{Y}| =
|\mathrm{ext}^{*}(-K_{Y'})-E|$. This implies that the map $f \circ
\mathrm{ext}^{-1}$ is given by the linear system $\big|-K_{Y'} - C
\big|$. On the other hand, only the following possibilities occur
due to Lemma~\ref{theorem:projection-when-curve-rem}:

\smallskip

\begin{itemize}

\item $E_{f'} \cap C = \emptyset$;

\smallskip

\item $f'(C)$ is a point and $C$ belongs to that component of
$E_{f'}$ which is mapped by $f'$ onto a curve;

\smallskip

\item $f'$ is a small contraction (near $C$) with $C\subseteq E_{f'}$.

\end{itemize}
Hence, as $K_{Y'} = f'^{*}(K_{X'})$ and $\mathrm{ext}^{-1}$ is
undefined precisely at $C$, the map $p$ is given by the linear
system $|-K_{X'}-f'(C)|$.
\end{proof}

\begin{lemma}
\label{theorem:projection-when-point} Let $P := \mathrm{ext}(E)$
be a smooth point on $Y'$. Then $p$ is the linear projection from
the tangent space at the smooth point $f'(P)$ on $X'$.
\end{lemma}

\begin{proof}
Recall that $\mathrm{ext}$ is the blowup of $Y'$ at $P$. In
particular, we have
\begin{equation}
\nonumber K_{Y} = \mathrm{ext}^{*}(K_{Y'})+2E.
\end{equation}
Thus $f$ is given by the linear system $|-K_{Y}| =
|\mathrm{ext}^{*}(-K_{Y'})-2E|$. This implies that the map $f
\circ \mathrm{ext}^{-1}$ is given by the linear system
$\big|-K_{Y'} - 2P \big|$. Notice also that $P \not\in E_{f'}$ by
Lemma~\ref{theorem:projection-when-curve-rem}. Hence, as $K_{Y'} =
f'^{*}(K_{X'})$ and $\mathrm{ext}^{-1}$ is undefined precisely at
$P$, the map $p$ is given by the linear system $|-K_{X'}-2f'(P)|$.
\end{proof}

\begin{lemma}
\label{theorem:degree-of-different-fanos} The inequality
$(-K_{Y'})^{3}
> (-K_{Y})^{3}$ holds.
\end{lemma}

\begin{proof}
By Proposition~\ref{theorem:class-of-extr-rays}, if
$\mathrm{ext}(E)$ is a point, then $(-K_{Y'})^{3}
> (-K_{Y})^{3}$. Further, if $\mathrm{ext}(E)$ is a curve, then inequality $(-K_{Y'})^{3}
> (-K_{Y})^{3}$ follows from Lemmas~\ref{theorem:projection-when-curve} and
\ref{theorem:bir-proj}.
\end{proof}

From Lemma~\ref{theorem:degree-of-different-fanos} and
Theorems~\ref{theorem:prokhorov-degree}, \ref{theorem:main-0} we
obtain that $X'$ is either
\begin{equation}
\nonumber \mathbb{P}(3,1,1,1) \qquad \mbox{or} \qquad
\mathbb{P}(6,4,1,1), \qquad \mbox{or} \qquad X_{70}, \qquad
\mbox{or} \qquad X_{66}.
\end{equation}

\begin{proposition}
\label{theorem:simple-case-1} Let $X' = \mathbb{P}(3,1,1,1)$. Then
$p$ is the linear projection from the tangent space at the smooth
point $f'(\mathrm{ext}(E))$ on $X'$. Conversely, for such $X'$ and
$p$, variety $X := p(X')$ is a Fano threefold with $(-K_{X})^{3} =
64$ (see $\eqref{num-1}$ in Theorem~\ref{theorem:main-1}).
\end{proposition}

\begin{proof}
Recall that $Y'$ is smooth and $E_{f'}$ is an irreducible divisor
contracted to a unique singular point on $X' =
\mathbb{P}(3,1,1,1)$ (see Example~\ref{example:examp-1} and
Remark~\ref{remark:unique-terminal-modification-1}). Then we have
the following:

\begin{lemma}
\label{theorme:not-a-curve-1} $\mathrm{ext}(E)$ is a point.
\end{lemma}

\begin{proof}
Suppose that $C := \mathrm{ext}(E)$ is a curve. Then $E_{f'} \cap
C = \emptyset$ by Lemma~\ref{theorem:projection-when-curve-rem}.
Further, the map $p: X' \dashrightarrow X$ is the linear
projection from a subspace $\Lambda \subset \mathbb{P}^{38}$ such
that $X'\cap\Lambda = f'(C)$ (see
Lemma~\ref{theorem:projection-when-curve}). More precisely, since
$g = 33$ and $g' = 37$, we find that $\dim \Lambda = 3$, which
gives $-K_{X'} \cdot f'(C) \leqslant 4$ (for $X'$ is an
intersection of quadrics). But on $X' = \mathbb{P}(3,1,1,1)$ we
have $\mathcal{O}_{X'}(-K_{X'})\simeq\mathcal{O}_{X'}(6)$ (see
\cite[Theorem 3.3.4]{Dolgachev}), and hence
$$
0 < \mathcal{O}_{X'}(1) \cdot f'(C) \leqslant \frac{2}{3},
$$
which implies that the curve $f'(C)$ passes through the singular
point on $X'$. This contradicts $E_{f'} \cap C = \emptyset$.
\end{proof}

Further, Lemmas~\ref{theorme:not-a-curve-1} and
\ref{theorem:projection-when-point} imply that $p$ is the linear
projection from the tangent space at the smooth point
$f'(\mathrm{ext}(E))$ on $X'$.

Conversely, linear projection from the tangent space at a smooth
point on $X' = \mathbb{P}(3,1,1,1)$ is birational by
Lemma~\ref{theorem:bir-proj} and leads to a Fano threefold $X :=
p(X')$ with $(-K_{X})^{3} = 64$ (cf.
Remark~\ref{remark:K-trivial-contraction-2} and the construction
in \cite[Section 7]{mukai}).
Proposition~\ref{theorem:simple-case-1} is completely proved.
\end{proof}

\begin{proposition}
\label{theorem:simple-case-2} Let $X' = \mathbb{P}(6,4,1,1)$. Then
$p$ is the linear projection from the tangent space at the smooth
$f'(\mathrm{ext}(E))$ on $X'$. Conversely, for such $X'$ and $p$,
variety $X := p(X')$ is a Fano threefold with $(-K_{X})^{3} = 64$
(see $\eqref{num-4}$ in Theorem~\ref{theorem:main-1}).
\end{proposition}

\begin{proof}
Recall that $Y'$ is smooth and the locus $E_{f'}$ is of pure
codimension $1$ for $X' = \mathbb{P}(6,4,1,1)$ (see
Example~\ref{example:examp-2} and
Remark~\ref{remark:unique-terminal-modification-2}). Then we have
the following two results:

\begin{lemma}
\label{theorme:not-a-curve-2} $f'(\mathrm{ext}(E))$ is a point.
\end{lemma}

\begin{proof}
Suppose that $f'(\mathrm{ext}(E))$ is a curve. Then $C :=
\mathrm{ext}(E)$ is also a curve and we have $E_{f'} \cap C =
\emptyset$ by Lemma~\ref{theorem:projection-when-curve-rem}.
Further, the map $p: X' \dashrightarrow X$ is the linear
projection from a subspace $\Lambda \subset \mathbb{P}^{38}$ such
that $X'\cap\Lambda = f'(C)$ (see
Lemma~\ref{theorem:projection-when-curve}). More precisely, since
$g = 33$ and $g' = 37$, we find that $\dim \Lambda = 3$, which
gives $-K_{X'} \cdot f'(C) \leqslant 4$ (for $X'$ is the
intersection of quadrics). But on $X' = \mathbb{P}(6,4,1,1)$ we
have $\mathcal{O}_{X'}(-K_{X'})\simeq\mathcal{O}_{X'}(12)$ (see
Example~\ref{example:examp-2}), and hence
$$
0 < \mathcal{O}_{X'}(1) \cdot f'(C) \leqslant \frac{1}{3},
$$
which implies that the curve $f'(C)$ passes through a singular
point on $X'$. This contradicts $E_{f'} \cap C = \emptyset$.
\end{proof}

\begin{lemma}
\label{theorem:when-e-is-point} $\mathrm{ext}(E)$ is a point.
\end{lemma}

\begin{proof}
Suppose that $C := \mathrm{ext}(E)$ is a curve. Then
Lemmas~\ref{theorme:not-a-curve-2} and
\ref{theorem:projection-when-curve} show that $p$ is the linear
projection from the point $f'(C)$. In particular, we get $g' - g =
1$, which contradicts $g = 33$, $g' = 37$.
\end{proof}

Further, Lemmas~\ref{theorem:when-e-is-point} and
\ref{theorem:projection-when-point} imply that $p$ is the linear
projection from the tangent space at the smooth point
$f'(\mathrm{ext}(E))$ on $X'$.

Conversely, linear projection from the tangent space at a smooth
point on $X' = \mathbb{P}(6,4,1,1)$ is birational (see
Lemma~\ref{theorem:bir-proj}) and leads to a Fano threefold $X :=
p(X')$ with $(-K_{X})^{3} = 64$ (cf.
Remark~\ref{remark:K-trivial-contraction-2} and the construction
in \cite[Section 7]{mukai}).
Proposition~\ref{theorem:simple-case-2} is completely proved.
\end{proof}

\begin{proposition}
\label{theorem:simple-case-3} Let $X' = X_{70}$. Then $p$ is the
linear projection from a plane $\Pi$. More precisely, $X'\cap\Pi=
f'(\mathrm{ext}(E))$ is a smooth conic that does not pass through
the singular point on $X'$. Conversely, for such $X'$, $\Pi$ and
$p$, variety $X := p(X')$ is a Fano threefold with $(-K_{X})^{3} =
64$. Moreover, there exists a smooth conic $C$ on $X'$ with $C
\cap \mathrm{Sing}(X') = \emptyset$ (see $\eqref{num-5}$ in
Theorem~\ref{theorem:main-1}).
\end{proposition}

\begin{proof}
Recall that $Y'$ is smooth, $E_{f'}$ is of pure codimension $1$,
and $f'(E_{f'})$ is the unique singular point on $X' = X_{70}$
(see Example~\ref{example:examp-5},
Remark~\ref{remark:unique-terminal-modification-5} and
Lemma~\ref{theorem:singularities-of-70}). Then we have the
following two results:

\begin{lemma}
\label{theorem:when-e-is-point-1} $\mathrm{ext}(E)$ is a curve.
\end{lemma}

\begin{proof}
Suppose that $P := \mathrm{ext}(E)$ is a point. Then $P\not\in
E_{f'}$ by Lemma~\ref{theorem:projection-when-curve-rem}.
Furthermore, Lemma~\ref{theorem:projection-when-point} implies
that $p$ is the linear projection from the tangent space at the
smooth point $f'(P)$ on $X'$. In particular, we get $g' - g = 4$,
which contradicts $g = 33$, $g' = 36$.
\end{proof}

\begin{lemma}
\label{theorem:when-e-is-curve-1} $f'(\mathrm{ext}(E))$ is a
curve.
\end{lemma}

\begin{proof}
Suppose that $P := f'(\mathrm{ext}(E))$ is a point. Then, since
$\mathrm{ext}(E)$ is a curve (see
Lemma~\ref{theorem:when-e-is-point-1}),
Lemma~\ref{theorem:projection-when-curve} implies that $p$ is the
linear projection from $P$. In particular, we get $g' - g = 1$,
which contradicts $g = 33$, $g' = 36$.
\end{proof}

Lemmas~\ref{theorem:when-e-is-curve-1} and
\ref{theorem:projection-when-curve} imply that $p$ is the linear
projection from a subspace $\Pi \subset \mathbb{P}^{37}$ such that
$X'\cap\Pi = f'(\mathrm{ext}(E)) =: C$. More precisely, since $g =
33$ and $g' = 36$, we find that $\dim \Pi = 2$, which gives
$-K_{X'} \cdot C \leqslant 2$ (for $X'$ is an intersection of
quadrics).

\begin{lemma}
\label{theorem:conic} $C$ is smooth and $-K_{X'} \cdot C = 2$,
i.e., $C$ is a conic on $X'$. We also have $C \cap
\mathrm{Sing}(X') = \emptyset$.
\end{lemma}

\begin{proof}
Recall that $E_{f'}$ is of pure codimension $1$ and $f'(E_{f'})$
is a point. Then we get $E_{f'} \cap \mathrm{ext}(E) = \emptyset$
by Lemma~\ref{theorem:projection-when-curve-rem}. Thus $C \cap
\mathrm{Sing}(X') = \emptyset$. Further, if $-K_{X'} \cdot C = 1$,
then $C$ is a line on $X$, which is impossible by
Proposition~\ref{theorem:contraction-to-curve-C}. In the same way
we exclude the case of singular $C$ with $-K_{X'} \cdot C = 2$.
\end{proof}

Lemma~\ref{theorem:conic} finishes the proof of the first part of
Proposition~\ref{theorem:simple-case-3}. Conversely, let $C$ be a
smooth conic on $X'$ such that $C \cap \mathrm{Sing}(X') =
\emptyset$. Consider the blowup $\rho : Y \longrightarrow X'$ of
$C$. Notice that $\rho$ resolves indeterminacies on $X'$ of the
linear projection from the plane $\Pi$ such that $X'\cap\Pi' = C$.
Then, by the choice of $C$, variety $Y$ is a weak Fano threefold
such that the linear system
$$
|-K_Y| = |\rho^{*}(-K_{X'}) - E_{\rho}|,\qquad \text{where}\qquad
E_{\rho} := \rho^{-1}(C),
$$
is basepoint-free and
$$
\big(-K_{Y}^{3}\big) = \big(\rho^{*}(-K_{X'}) - E_{\rho}\big)^3 =
\big(-K_{X'}^{3}\big) - 3K_{X'} \cdot E_{\rho}^{2} - E_{\rho}^{2}
= 70 + 3K_{X'} \cdot C -K_{X'}\cdot C - 2 = 64.
$$
The morphism $f : = \Phi_{|-K_{Y}|} : Y \longrightarrow f(Y) =: X$
is birational (see Lemma~\ref{theorem:bir-proj}), and hence $X$ is
a Fano threefold with $(-K_{X})^{3} = 64$ (cf.
Remark~\ref{remark:K-trivial-contraction-2}).

Finally, let us prove the following:

\begin{lemma}
\label{theorem:conic-1} There exists a smooth conic $C$ on $X'$
such that $C \cap \mathrm{Sing}(X') = \emptyset$.
\end{lemma}

\begin{proof}
In the notation from the proof of
Lemma~\ref{theorem:non-simple-case}, the surface $S :=
\tau(E_{\sigma})$ is the quadratic cone on $X'$, i.e., $S$ is a
minimal surface of degree $2$, singular at the point
$\mathrm{Sing}(X')$. We have $S \subseteq X' \cap \mathbb{P}^3
\subset \mathbb{P}^{37}$ because $X'$ is an intersection of
quadrics. Then generic hyperplane section of $S$ is a smooth conic
$C$ on $X'$ such that $C \cap \mathrm{Sing}(X')$.
\end{proof}

Lemma~\ref{theorem:conic-1} finishes the proof of
Proposition~\ref{theorem:simple-case-3}.
\end{proof}

\begin{proposition}
\label{theorem:simple-case-4} Let $X' = X_{66}$. Then $p$ is the
linear projection from a singular $\mathrm{cDV}$ point on
$X_{66}$.Conversely, for such $X'$ and $p$, variety $X := p(X')$
is a Fano threefold with $(-K_{X})^{3} = 64$ (see $\eqref{num-2}$
in Theorem~\ref{theorem:main-1}).
\end{proposition}

\begin{proof}
Recall that $Y'$ is smooth for $X' = X_{66}$ (see
Remark~\ref{remark:unique-terminal-modification-6}). Then have the
following two results:

\begin{lemma}
\label{theorem:when-e-is-point-2} $\mathrm{ext}(E)$ is a curve.
\end{lemma}

\begin{proof}
Suppose that $\mathrm{ext}(E)$ is a point. Then, repeating the
arguments in the proof of Lemma~\ref{theorem:when-e-is-point-1}
verbatim we get $g' - g = 4$, which contradicts $g = 33$, $g' =
34$.
\end{proof}

\begin{lemma}
\label{theorem:when-e-is-curve-2} $f'(\mathrm{ext}(E))$ is a
point.
\end{lemma}

\begin{proof}
Suppose that $C := f'(\mathrm{ext}(E))$ is a curve. Then, since
$\mathrm{ext}(E)$ is a curve (see
Lemma~\ref{theorem:when-e-is-point-2}),
Lemma~\ref{theorem:projection-when-curve} implies that $p$ is a
linear projection from a subspace $\Lambda \subset
\mathbb{P}^{35}$ such that $X' \cap \Lambda = C$. In particular,
we have $g' - g \geqslant 2$ (see \eqref{deg-estimate}), which
contradicts $g = 33$, $g' = 34$.
\end{proof}

Lemmas~\ref{theorem:projection-when-curve} and
\ref{theorem:when-e-is-curve-2} imply that $p$ is the linear
projection from the point $P := f'(\mathrm{ext}(E))$. Notice that
$P \in \mathrm{Sing}(X')$, since otherwise we get $(-K_{X})^{3} =
65$, which is impossible by \eqref{deg-estimate}. Moreover, it
follows from Lemma~\ref{theorem:only-cdv-points} that $P$ is a
singular $\mathrm{cDV}$ point on $X'$. Thus, according to
Lemma~\ref{theorem:non-simple-case}, it remains to prove the
following:

\begin{lemma}
\label{theorem:contraction-to-curve-64-2} $X'$ contains a singular
$\mathrm{cA_{1}}$ point.
\end{lemma}

\begin{proof}
$X'$ is a toric variety (see Example~\ref{example:examp-6}). It is
given by the fan $\Sigma\subset N \otimes_{\mathbb{Z}}
\mathbb{R}\simeq \mathbb{R}^{3}$ generated by the vectors
$$
e_{1} := (-1,0,0), \ e_{2} := (1,-1,0), \ e_{3} := (-1,-1,2), \
e_{4} := (-1,-1,3), \ e_{5} := (-1, 2, -1)
$$
in the standard lattice $N := \mathbb{Z}^{3}$ in $\mathbb{R}^{3}$
(see \cite{Kreuzer-Skarke}). Then the affine
$(\mathbb{C}^{*})^{3}$-invariant cover of $X'$ is determined by
decomposition of $\Sigma$ into the following cones (cf.
\cite{fulton}):
$$
\Sigma_{1} := <e_{1},e_{2},e_{3}>, \ \Sigma_{2} := <e_{1}, e_{3},
e_{4}, e_{5}>, \ \Sigma_{3} := <e_{2}, e_{3}, e_{4}, e_{5}>, \
\Sigma_{4} := <e_{1}, e_{2}, e_{5}>
$$
\begin{figure}
\includegraphics[scale=1.5]{test1.1}
\caption{}\label{fig-1}
\end{figure}
(see Figure \ref{fig-1}).

Calculating the volume of the simplex spanned by the vectors
$e_{1},e_{2},e_{3}$, we obtain that the sublattice in $N$
generated by $e_{1},e_{2},e_{3}$ has index $2$. Hence $\Sigma_{1}$
determines a singular affine chart
$U_{\scriptscriptstyle\Sigma_{1}}$, isomorphic to
$\mathbb{C}^{3}/\mu_2$ for a linear action on $\mathbb{C}^3$ of
the cyclic group $\mu_2$ of order $2$ (see \cite[2.1,
2.2]{fulton}). Moreover, singularities of
$U_{\scriptscriptstyle\Sigma_{1}}$ are non-isolated, for otherwise
$X'$ has a unique non-Gorenstein singularity in
$U_{\scriptscriptstyle\Sigma_{1}}$ of type
$\displaystyle\frac{1}{2}(1,1,1)$, which is impossible. Thus we
find that
$$
U_{\scriptscriptstyle\Sigma_{1}} \simeq \mathbb{C} \times
\big(\mathbb{C}^{2}/\mathbb{Z}_{2}\big),
$$
and hence $X'$ contains a singular $\mathrm{cA_{1}}$ point.
\end{proof}

Lemma~\ref{theorem:contraction-to-curve-64-2} finishes the proof
of Proposition~\ref{theorem:simple-case-4}.
\end{proof}

\bigskip

\section{General case: reduction to the log Mori fibration}
\label{section:mori-fibration-red}

Let $X$, $Y$ and $f$ be as in
Proposition~\ref{theorem:terminal-modification}. Through the rest
of the present Section we consider Fano threefold $X$ such that
$(-K_{X}^{3}) = 64$. Again, as in
Section~\ref{section:special-type-contraction}, we identify $X$
with its anticanonical image in $\mathbb{P}^{34}$.

It follows from Propositions~\ref{theorem:namikawa-smoothing},
\ref{theorem:class-of-extr-rays}, \ref{theorem:1-contraction} and
results of Sections~\ref{section:mori-fiber-space},
\ref{section:special-type-contraction} that in order to prove
Theorem~\ref{theorem:main-1} one may reduce to the case when
$\mathrm{ext}: Y \longrightarrow Y'$ is a birational contraction
to either weak Fano threefold $Y'$ with non-Gorenstein
singularities, or to a factorial threefold $Y'$ such that divisor
$-K_{Y'}$ is not nef. Then, according to
Corollaries~\ref{theorem:0-cor-contraction} and
\ref{theorem:1-cor-contraction}, we may assume that $X =
\Phi_{\scriptscriptstyle |-K_{X}|}(X) \subset \mathbb{P}^{34}$
satisfies one of the following conditions:
\begin{itemize}

\item $X$ is singular along a line $\Gamma$ (case ${\bf A}$);

\item $X$ contains a plane $\Pi$ (case ${\bf B}$);

\item $X$ contains a non-$\mathrm{cDV}$ point $O$ (case ${\bf C}$).

\end{itemize}

Set $\mathcal{L}: = |-K_{X}|$ and consider the following linear
systems:

\begin{itemize}

\item $\mathcal{H}:=\{H\in \mathcal{L} \ | \ H\supset \Gamma\}$ in case ${\bf A}$;

\item $\mathcal{H}:=\{H \ | \ H+\Pi\in\mathcal{L}\}$ in case ${\bf B}$;

\item $\mathcal{H}:=\{H\in \mathcal{L} \ | \ H \ni O\}$ in case ${\bf C}$.

\end{itemize}

Take $f: Y \longrightarrow X$ as in
Section~\ref{section:preliminaries}. Put $\mathcal{L}_{Y} :=
f_{*}^{-1}(\mathcal{L})$ and $\mathcal{H}_{Y} :=
f_{*}^{-1}(\mathcal{H})$. Then for generic element $H \in
\mathcal{H}$ and $H_{Y} := f_{*}^{-1}(H) \in \mathcal{H}_{Y}$, we
have
\begin{eqnarray}
\label{some-equalities-1} K_{Y}+H_{Y}+D_{Y}=f^{*}(K_{X}+H)\sim
0\qquad\mbox{in cases ${\bf A}$ and ${\bf C}$}
\end{eqnarray}
and
\begin{eqnarray}
\label{some-equalities-2} K_{Y}+H_{Y}+D_{Y}=f^{*}(K_{X}+H+\Pi)\sim
0\qquad\mbox{in case ${\bf B}$},
\end{eqnarray}
where $D_{Y}$ is an effective non-zero integral $f$-exceptional
divisor in cases ${\bf A}$ and ${\bf C}$, and the sum of
$f_{*}^{-1}(\Pi)$ and an effective integral $f$-exceptional
divisor in case ${\bf B}$.

On the other hand, for generic element $L \in \mathcal{L}$ and
$L_{Y} := f_{*}^{-1}(L) \in \mathcal{L}_{Y}$ we have
\begin{equation}
\label{for-lin-sys-on-y} K_{Y} + L_{Y} = f^{*}(K_{X}+L) \sim 0.
\end{equation}

Let us state some properties of the objects just introduced (see
\cite[\S 6]{Prokhorov-degree} for the proofs):

\begin{lemma}[see {\cite[Lemma 6.3]{Prokhorov-degree}}]
\label{theorem:image-of-projection} The image of the threefold $X$
under the map $\varPhi_{\mathcal{H}}$ is three-dimensional.
\end{lemma}

\begin{lemma}[see {\cite[Corollary 6.4]{Prokhorov-degree}}]
\label{theorem:can-pair-on-y} The pair $(Y, L_{Y})$ has at worst
canonical singularities.
\end{lemma}

\begin{lemma}[see {\cite[Lemma 6.6]{Prokhorov-degree}}]
\label{theorem:good-properties} $Y$ can be taken in such a way
that

\begin{itemize}

\item the pair $(Y, H_{Y})$ has at worst canonical singularities;

\item the linear system $\mathcal{H}_{Y}$ consists of nef Cartier divisors.

\end{itemize}

\end{lemma}

We assume in what follows that $Y$ satisfies the conditions of
Lemma~\ref{theorem:good-properties}. Further, let us apply the log
MMP to the pair $(Y, H_{Y})$. Then on each step the identity $K +
\mathcal{H} \equiv -D$, for some effective integral divisor $D \ne
0$, is preserved (cf. \eqref{some-equalities-1} and
\eqref{some-equalities-2}). Hence at the end of the program one
arrives at a pair $(W, H_{W})$ with a $(K_{W}+H_{W})$-negative
extremal contraction $\mathrm{ext}_{\scriptscriptstyle W}: W
\longrightarrow V$ to a lower-dimensional variety $V$.

\begin{remark}
\label{remark:singularities-of-pairs} Let $\mathcal{H}_{W}$ be the
proper transform on $W$ of the linear system $\mathcal{H}_{Y}$.
Then, since $H_{W} \in \mathcal{H}_{W}$,
Lemma~\ref{theorem:good-properties} and \cite[Lemma
3.4]{Prokhorov-degree} imply that the pair $(W, H_{W})$ is
canonical and the linear system $\mathcal{H}_{W}$ consists of nef
Cartier divisors. In particular, $\mathrm{ext}_{\scriptscriptstyle
W}$ is a $K_{W}$-negative extremal contraction and $W$ has at
worst terminal $\mathbb{Q}$-factorial singularities. Furthermore,
let $\mathcal{L}_{W}$ be the proper transform on $W$ of the linear
system $\mathcal{L}_{Y}$. Then for generic element $L_{W} \in
\mathcal{L}_{W}$, we have $K_{W} + L_{W} \equiv 0$ (cf.
\eqref{for-lin-sys-on-y}), and Lemma~\ref{theorem:can-pair-on-y}
together with \cite[Lemma 3.1]{Prokhorov-degree} imply that the
pair $(W, L_{W})$ has at worst canonical singularities. In
particular, since $\mathcal{L}_{W} \subseteq |-K_{W}|$, the linear
system $|-K_{W}|$ does not have fixed components. Finally,
$(W,L_{W})$ is a \emph{generating $0$-pair} (see \cite[Definition
4.1]{Prokhorov-degree} and \cite{shok-prok-comp}).
\end{remark}

\begin{remark}
\label{remark:linear-systems} By construction, the initial
threefold $X$ is the image of $W$ under the birational map
$\Phi_{\mathcal{L}_{W}}$. Furthermore, it follows from
Lemma~\ref{theorem:image-of-projection} that divisor $H_{W}$ is
$\mathrm{ext}_{\scriptscriptstyle W}$-ample and the linear system
$\mathcal{H}_{W}$ does not have fixed components.
\end{remark}

Let us stress once again that
\begin{equation}
\label{some-equality-on-w} K_{W}+H_{W}+D_{W}\equiv 0
\end{equation}
for some effective integral $D_{W} \ne 0$.

\begin{remark}
\label{remark:inequalities} Note that $\dim |-K_{W}| \geqslant
\dim\mathcal{L}_{W} = \dim |-K_{X}|$ (see
Remarks~\ref{remark:singularities-of-pairs},
\ref{remark:linear-systems}). Moreover, by construction of
$\mathcal{H}$ the estimate $\dim\mathcal{H} \geqslant \dim
|-K_{X}| - 3$ takes place. It then follows from $(-K_{X})^{3} =
64$ that $\dim |-K_{W}| \geqslant 34$ and $\dim |H_{W}| \geqslant
31$.
\end{remark}

We conclude this Section by proving the following:

\begin{lemma}
\label{theorem:can-pair-on-y-01} Let $\dim V = 0$. Then $X =
\mathbb{P}^{3}$.
\end{lemma}

\begin{proof}
Under the stated assumption, $W$ is a $\mathbb{Q}$-Fano threefold.
Then \cite[Proposition 7.2]{Prokhorov-degree} implies that $\dim
|-K_{W}| \leqslant 34$ and we get $|-K_{W}| = \mathcal{L}_{W}$
according to Remark~\ref{remark:inequalities}. Furthermore, $W$
has at worst Gorenstein singularities in our case, since $\dim
|-K_{X}| = 34$ (see \cite[Section 7]{Prokhorov-degree}). Then we
obtain $W = Y = X = \mathbb{P}^3$ by
Proposition~\ref{theorem:namikawa-smoothing}.
\end{proof}

\bigskip

\section{Continuation of the proof of Theorem~\ref{theorem:main-1}: the case of contraction to a curve}
\label{section:case-curve-contr}

Let $\mathrm{ext}_{\scriptscriptstyle W}: W \longrightarrow V$ be
the extremal contraction from
Section~\ref{section:mori-fibration-red}. For the rest of the
present Section we assume that $\dim V = 1$. The next statement is
evident (cf. \cite[Section 5]{karz}):

\begin{proposition}
\label{theorem:special-scrolls} The following hold:

\begin{itemize}

\item $V \simeq \mathbb{P}^{1}$;

\item the general fibre $W_{\eta}$ of the morphism
$\mathrm{ext}_{\scriptscriptstyle W}: W \longrightarrow V$ is
isomorphic to either $\mathbb{P}^{2}$ or
$\mathbb{P}^{1}\times\mathbb{P}^{1}$.

\end{itemize}

\end{proposition}

\begin{proposition}
\label{theorem:special-scrolls-1} If $W_{\eta} \simeq
\mathbb{P}^{2}$, then one of the following holds:

\begin{itemize}

\item $X \subset \mathbb{P}^{34}$ is the image of the threefold
$\mathbb{P}(6,4,1,1) \subset \mathbb{P}^{38}$ under birational
linear projection;

\item $X \subset \mathbb{P}^{34}$ is the image of the threefold
$X_{70} \subset \mathbb{P}^{37}$ under birational linear
projection;

\item $X \subset \mathbb{P}^{34}$ is the image of the threefold
$X_{66} \subset \mathbb{P}^{35}$ under birational linear
projection.

\end{itemize}

\end{proposition}

\begin{proof}
Recall that $\dim |-K_W|\geqslant 34$ (see
Remark~\ref{remark:inequalities}). Then the proof of \cite[Lemma
5.1]{karz} leaves us with the only option for $W$. More precisely,
$W$ must be one of the two rational scrolls, with the
anticanonical maps for both being birational and mapping $W$ onto
either $\mathbb{P}(6,4,1,1)$ (cf. \cite[Chapter 4, Remark
4.2]{Iskovskikh-anti-canonical-models}) or $X_{66}$ (cf.
\cite[Proposition 5.2]{karz}). Now, since $X_{70}$ is the image of
$\mathbb{P}(6,4,1,1)$ under the linear projection from a point,
Proposition~\ref{theorem:special-scrolls-1} follows again from the
construction of $\mathcal{L}_W$: the rational map $W
\dashrightarrow X$, given by $\mathcal{L}_W\subseteq |-K_W|$ (see
Remark~\ref{remark:linear-systems}), yields a linear projection of
the image $\Phi_{\scriptscriptstyle |-K_W|}(W)$ onto $X$.
\end{proof}

\begin{proposition}
\label{theorem:projection-of-X-66} Let $X \subset \mathbb{P}^{34}$
be the image of the threefold $X_{66} \subset \mathbb{P}^{35}$
under birational linear projection. Then $X$ is the threefold
constructed in Proposition~\ref{theorem:simple-case-4}.
\end{proposition}

\begin{proof}
Let $\pi: X_{66} \dashrightarrow X$ be the given projection.
Notice that $\pi$ is the projection from a point $P$ (see
\eqref{deg-estimate}). We also have $P \in \text{Sing}(X_{66})$,
for otherwise we would get $(-K_X)^3 = 65$, a contradiction.
Further, the singularity $(P \in X_{66})$ is $\mathrm{cDV}$ by
Lemma~\ref{theorem:only-cdv-points}, and
Proposition~\ref{theorem:projection-of-X-66} follows.
\end{proof}

\begin{proposition}
\label{theorem:projection-of-X-70} Let $X \subset \mathbb{P}^{34}$
be the image of the threefold $X_{70} \subset \mathbb{P}^{37}$
under birational linear projection. Then $X$ is the threefold
constructed in Proposition~\ref{theorem:simple-case-3}.
\end{proposition}

\begin{proof}
Let $\pi: X_{70} \dashrightarrow X$ be the given projection. Let
us also denote by $\Pi\subset\mathbb{P}^{37}$ the subspace that
$\pi$ projects from. Then $\dim\Pi = 2$ by \eqref{deg-estimate}
and $X_{70} \cap \Pi \ne \emptyset$ by
Lemma~\ref{theorem:bir-proj}.

\begin{lemma}
\label{theorem:intersection-is-not-a-plane-70-1} The scheme
$X_{70} \cap \Pi$ is a smooth conic.
\end{lemma}

\begin{proof}
Since $X_{70} \subset \mathbb{P}^{37}$ is an intersection of
quadrics and $\dim\Pi = 2$, the locus $X_{70} \cap \Pi$ is either
a curve of degree $\leqslant 2$ ($X_{70} \cap \Pi$ need not be of
pure dimension $1$ here) or consists (as a set) of $\leqslant 4$
points. But the latter case is impossible by
Lemmas~\ref{theorem:use-pro-l} and
\ref{theorem:singularities-of-70}, \ref{theorem:only-cdv-points}.
Finally, if $X_{70} \cap \Pi$ is a curve, then
Proposition~\ref{theorem:contraction-to-curve-C} and
Lemmas~\ref{theorem:singularities-of-70},
\ref{theorem:only-cdv-points} immediately imply that $X_{70} \cap
\Pi$ is a conic (and so $X_{70} \cap \Pi$ is of pure dimension $1$
because its degree equals $2$ in this case).
\end{proof}

Proposition~\ref{theorem:projection-of-X-70} now follows
Lemmas~\ref{theorem:intersection-is-not-a-plane-70-1},
\ref{theorem:singularities-of-70} and
\ref{theorem:only-cdv-points}.
\end{proof}

\begin{proposition}
\label{theorem:projection-of-P-64} Let $X$ be the image of the
threefold $\mathbb{P}(6,4,1,1) \subset \mathbb{P}^{38}$ under
birational linear projection. Then one of the following holds:

\begin{itemize}

\item $X$ is the threefold constructed in
Proposition~\ref{theorem:simple-case-2};

\item $X$ is the threefold constructed in
Proposition~\ref{theorem:simple-case-3}.

\end{itemize}

\end{proposition}

\begin{proof}
Put $\mathbb{P} := \mathbb{P}(6,4,1,1)\subset \mathbb{P}^{38}$ and
let $\pi: \mathbb{P} \dashrightarrow X$ be the given projection.
Let us also denote by $\Omega\subset\mathbb{P}^{38}$ the subspace
that $\pi$ projects from. Then $\dim\Omega = 3  $ by
\eqref{deg-estimate} and $\mathbb{P} \cap \Omega \ne \emptyset$ by
Lemma~\ref{theorem:bir-proj}. We have the following three results:

\begin{lemma}
\label{theorem:intersection-is-not-a-plane-1} The scheme
$\mathbb{P} \cap \Omega$ either contains a curve of degree
$\leqslant 4$ or consists (as a set) of $\leqslant 8$ points.
\end{lemma}

\begin{proof}
Since $\mathbb{P} \subset \mathbb{P}^{38}$ is an intersection of
quadrics and $\dim\Omega = 3$, the locus $\mathbb{P} \cap \Omega$
either contains a curve of degree $\leqslant 4$, or it is a
quadric surface in $\mathbb{P}^3$, or consists of $\leqslant 8$
points. But the second (quadric) option does occur by
Lemma~\ref{theorem:singularities-of-X}.
\end{proof}

\begin{lemma}
\label{theorem:intersection-is-not-a-plane-2} Let the scheme
$\mathbb{P} \cap \Omega$ contain a curve and $\dim (\mathbb{P}
\cap \Omega) = 1$. Then $X$ is the threefold constructed in
Proposition~\ref{theorem:simple-case-3}.
\end{lemma}

\begin{proof}
Take an irreducible curve $Z \subset \mathbb{P} \cap \Omega$. Then
we have $\deg (Z) \leqslant 4$ (see
Lemma~\ref{theorem:intersection-is-not-a-plane-1}). In particular,
we get
\begin{equation}
\nonumber \mathcal{O}_{X}(1) \cdot L_0 \leqslant \frac{1}{3},
\end{equation}
which implies that $Z$ passes through a singular point $P$ on
$\mathbb{P}$. Hence $\Omega \cap \mathrm{Sing}(\mathbb{P}) \ni P$.

Notice that singularity $(P \in \mathbb{P})$ is $\mathrm{cDV}$
(see Lemma~\ref{theorem:only-cdv-points}). Then it follows from
the construction of $X_{70}$ in Example~\ref{example:examp-5} that
$X$ is the image of $X_{70} \subset \mathbb{P}^{37}$ under
birational linear projection. Now
Proposition~\ref{theorem:projection-of-X-70} finishes the proof.
\end{proof}

\begin{lemma}
\label{theorem:intersection-is-not-a-plane-3} Let the scheme
$\mathbb{P} \cap \Omega$ be finite. Then $\Omega \cap
\mathrm{Sing}(\mathbb{P}) = \emptyset$.
\end{lemma}

\begin{proof}
Suppose that $\mathbb{P} \cap \Omega$ is a finite scheme and
$\mathbb{P} \cap \Omega \cap \mathrm{Sing}(\mathbb{P}) \ne
\emptyset$. Then, as in the proof of
Lemma~\ref{theorem:intersection-is-not-a-plane-2}, we obtain that
$X$ is the image of $X_{70} \subset \mathbb{P}^{37}$ under
birational linear projection from a plane $\Pi$ such that $X_{70}
\cap \Pi$ is a conic (cf.
Proposition~\ref{theorem:projection-of-X-70}). On the other hand,
since $\mathbb{P} \cap \Omega$ is a finite set, it follows from
the construction of $X_{70}$ in Example~\ref{example:examp-5} that
the set $X_{70} \cap \Pi$ must also be finite, a contradiction.
\end{proof}

It follows from Lemmas~\ref{theorem:use-pro-l} and
\ref{theorem:intersection-is-not-a-plane-3} that once the scheme
$\mathbb{P} \cap \Omega$ is finite, it does not intersect
$\mathrm{Sing}(\mathbb{P})$ and it can not be reduced. Thus, since
$\mathbb{P}$ is an intersection of quadrics, $\Omega$ coincides
with the tangent space at a smooth point on $\mathbb{P}$, provided
that $\mathbb{P} \cap \Omega$ is finite. This and
Lemma~\ref{theorem:intersection-is-not-a-plane-2} prove
Proposition~\ref{theorem:projection-of-P-64}.
\end{proof}
Propositions~\ref{theorem:special-scrolls-1},
\ref{theorem:projection-of-X-66}, \ref{theorem:projection-of-X-70}
and \ref{theorem:projection-of-P-64} provide a complete
description of those $X$ for which $W_{\eta} \simeq
\mathbb{P}^{2}$ (cf. Proposition~\ref{theorem:special-scrolls}).
Let us now turn to the second option when $W_{\eta} \simeq
\mathbb{P}^{1}\times\mathbb{P}^{1}$. In this case, there exists an
embedding $W \hookrightarrow \mathbb{F}$ over $V$, where
$$
\mathbb{F} := \mathbb{P}_{V}(\mathcal{E}), \qquad \mathcal{E} :=
\bigoplus _{i=1}^{4} \mathcal{O}_{\mathbb{P}^{1}}(d_{i}),
$$
$d_{1} \geqslant d_{2} \geqslant d_{3} \geqslant d_{4} = 0$, such
that $W_{\eta} \subset \mathbb{F}_{\eta} \simeq \mathbb{P}^{3}$ is
a smooth quadric (see \cite[Proposition 9.2]{Prokhorov-degree}).
Let $M$ and $F$ be tautological divisor and a fibre, respectively,
of the $\mathbb{P}^{3}$-bundle $\mathbb{F} \to V =
\mathbb{P}^{1}$. We have $W \sim 2M + rF$ for some $r \in
\mathbb{Z}$. Furthermore, from \cite[Lemma 9.5]{Prokhorov-degree}
we get
\begin{equation}
\label{res-of-adj-on-w} -K_{W} = 2G + (2 - d - r)N
\end{equation}
by adjunction, where $G := M\big\vert_{W}$, $N := F
\big\vert_{W}$, $d := \sum_{i=1}^4 d_i$.

\begin{proposition}
\label{theorem:some-est-from-prok} The inequality $d + r \geqslant
3$ holds.
\end{proposition}

\begin{proof}

Suppose that $d + r < 3$.

\begin{lemma}
\label{theorem:some-est-from-prok-1} $d + r \geqslant 2$.
\end{lemma}

\begin{proof}
Suppose that $d + r < 2$. Then the divisor $-K_{W}$ is ample (see
\eqref{res-of-adj-on-w}). Recall that $W$ has at worst terminal
singularities (see Remark~\ref{remark:singularities-of-pairs}).
Moreover, these singularities are Gorenstein (see
\eqref{res-of-adj-on-w}), which implies that $(-K_{W})^3 \leqslant
64$ (see Remark~\ref{remark:namikawa-smoothing-rem}). Then we get
$\dim |-K_{W}| \leqslant 34$ by \eqref{deg-estimate}. Furthermore,
it follows from Remark~\ref{remark:inequalities} that $\dim
|-K_{W}| = 34$, and hence $(-K_{W})^3 = 64$. Then we get $W =
\mathbb{P}^3$ (see Remark~\ref{remark:namikawa-smoothing-rem}),
which is impossible because $\rho(W) = 2$ by construction.
\end{proof}

It follows from Lemma~\ref{theorem:some-est-from-prok-1} that $d +
r = 2$. Then $-K_{W} = 2G$ (see \eqref{res-of-adj-on-w}). In
particular, $-K_{W}$ is a nef and big Cartier divisor, and from
the Riemann-Roch formula and Kawamata-Viehweg vanishing theorem we
obtain
\begin{equation}
\label{formula-for-deg-w} \dim \big|-K_{W}\big| =
-\frac{1}{2}K_{W}^{3} + 2
\end{equation}
(recall that $W$ has at worst terminal singularities). Then
Remark~\ref{remark:inequalities} gives
\begin{equation}
\label{estimate-for-deg-of-w} \big(-K_{W}\big)^{3} \in \{64, 66,
68, 70, 72\}
\end{equation}
(cf. Remark~\ref{remark:K-trivial-contraction-1} and
Theorems~\ref{theorem:prokhorov-degree}, \ref{theorem:main-0}).

\begin{lemma}
\label{theorem:some-est-from-prok-2} $(-K_{W})^{3} \ne 72$.
\end{lemma}

\begin{proof}
Let $(-K_{W})^{3} = 72$. Recall that $W$ has at worst terminal
$\mathbb{Q}$-factorial Gorenstein singularities. Hence $W$ is a
terminal $\mathbb{Q}$-factorial modification of either
$\mathbb{P}(3,1,1,1)$ or $\mathbb{P}(6,4,1,1)$. In particular, $W$
is isomorphic to one of the threefolds constructed in
Examples~\ref{example:examp-1}, \ref{example:examp-2} (cf.
Remarks~\ref{remark:unique-terminal-modification-1},
\ref{remark:unique-terminal-modification-2}). On the other hand,
we have $\rho(W) = 2$ by construction, which implies that $W$ is a
terminal $\mathbb{Q}$-factorial modification of
$\mathbb{P}(3,1,1,1)$. Then $W$ contains a surface $S$ covered by
$K_{W}$-trivial curves and such that $S \cap W_{\eta}$ is also a
$K_{W}$-trivial curve. But the latter is impossible because the
curves in the fibres of $\mathrm{ext}_{\scriptscriptstyle W}$ are
numerically proportional (for $\rho(W) = 2$) and have negative
intersection with $K_{W}$ (for the contraction
$\mathrm{ext}_{\scriptscriptstyle W}$ is $K_W$-negative by
definition).
\end{proof}

\begin{lemma}
\label{theorem:some-est-from-prok-3} $(-K_{W})^{3} \not\in \{66,
68, 70\}$.
\end{lemma}

\begin{proof}
Suppose that $(-K_{W})^{3} \in \{66, 68, 70\}$. Notice that
generic element in $|G|$ has at worst Du Val singularities (see
\cite{Shin}), which gives $K_{G}^{2} =
\displaystyle\frac{1}{8}(-K_{W})^{3} \not\in \mathbb{Z}$ by the
adjunction formula, a contradiction.
\end{proof}

\begin{lemma}
\label{theorem:some-est-from-prok-4} $(-K_{W})^{3} \ne 64$.
\end{lemma}

\begin{proof}
Suppose that $(-K_{W})^{3} = 64$. Then
Remark~\ref{remark:inequalities} and \eqref{formula-for-deg-w}
imply that $\mathcal{L}_{W} = |-K_{W}|$. On the other hand, since
$\rho(W) = 2$, the restriction to $W$ of the negative section of
the $\mathbb{P}^{3}$-bundle $\mathbb{F} \supset W$ is a curve on
$W$ (see the arguments in the proof of \cite[Proposition
9.4]{Prokhorov-degree}). Then the identity $-K_{W} = 2G$ shows
that $\Phi_{\mathcal{L}_{W}} : W \longrightarrow X$ is a small
contraction such that $K_{W} = \Phi_{\mathcal{L}_{W}}^{*}(K_{X})$.
Thus $X$ has at worst terminal Gorenstein singularities, and
Remark~\ref{remark:namikawa-smoothing-rem} yields $X = Y =
\mathbb{P}^{3}$, which is impossible.
\end{proof}

From \eqref{estimate-for-deg-of-w} and
Lemmas~\ref{theorem:some-est-from-prok-1},
\ref{theorem:some-est-from-prok-2},
\ref{theorem:some-est-from-prok-3},
\ref{theorem:some-est-from-prok-4} we get contradiction, and
Proposition~\ref{theorem:some-est-from-prok} is completely proved.
\end{proof}

It follows from Proposition~\ref{theorem:some-est-from-prok} and
the arguments in the proof of \cite[Proposition
9.4]{Prokhorov-degree} that there exists a birational map $W
\dashrightarrow W_{0}$ onto a $\mathbb{P}^{2}$-bundle $W_{0}$ over
$V$. Furthermore, threefold $W_{0}$ and the proper transforms
$\mathcal{L}_{W_{0}}$ and $\mathcal{H}_{W_{0}}$ on $W_{0}$ of the
linear systems $\mathcal{L}_{W}$ and $\mathcal{H}_{W}$,
respectively, possess the same properties as $W$ and
$\mathcal{L}_{W}$, $\mathcal{H}_{W}$ (see
\eqref{some-equality-on-w} and
Remarks~\ref{remark:singularities-of-pairs},
\ref{remark:linear-systems}, \ref{remark:inequalities}). Hence we
arrive at the setting of
Proposition~\ref{theorem:special-scrolls-1}. We can then repeat
the arguments from the proof of
Propositions~\ref{theorem:projection-of-X-66},
\ref{theorem:projection-of-X-70} and
\ref{theorem:projection-of-P-64} once again to find that $X$ is
just one of the threefolds constructed in
Propositions~\ref{theorem:simple-case-2},
\ref{theorem:simple-case-3} and \ref{theorem:simple-case-4}. This
completely sets up the case when $\dim V = 1$.

\bigskip

\section{End of the proof of Theorem~\ref{theorem:main-1}: the case of contraction to a surface}
\label{section:contraction-to-surface-case}

Let $\mathrm{ext}_{\scriptscriptstyle W} : W \longrightarrow V$ be
the extremal contraction from
Section~\ref{section:mori-fibration-red}.
Lemma~\ref{theorem:can-pair-on-y-01} and results of
Section~\ref{section:case-curve-contr} show that to finish the
proof of Theorem~\ref{theorem:main-1} it remains to treat the last
case when $\dim V = 2$.

Recall that $V$ is a smooth rational surface and
$\mathrm{ext}_{\scriptscriptstyle W} : W \longrightarrow V$ is a
$\mathbb{P}^1$-bundle (see \cite[Section 10]{Prokhorov-degree}).
Consider the rank $2$ vector bundle
\begin{equation}
\label{equation:vector-bundle} \mathcal{E} :=
\mathrm{ext}_{{\scriptscriptstyle W}*}(\mathcal{O}_{W}(H_{W}))
\end{equation}
on $V$. Then we have $W = \mathbb{P}(\mathcal{E})$. Furthermore,
since the divisor $H_W$ is nef and $(W,L_{W})$ is a generating
$0$-pair (see Remark~\ref{remark:singularities-of-pairs}), from
\cite[Lemma 4.4]{Prokhorov-degree} we deduce that
\begin{equation}
\label{equation:dimension-of-H}
\begin{array}{c}
\dim |H_{W}| + 1 = h^{0}(V, \mathcal{E}) = \chi(V, \mathcal{E}).
\end{array}
\end{equation}

\begin{lemma}
\label{theorem:estimate-for-h-64} One of the following holds:

\smallskip

\begin{itemize}

\item $\dim |H_{W}| \in \{31,32,33,34,35\}$;

\smallskip

\item $X$ is the threefold constructed in
Proposition~\ref{theorem:simple-case-1}.

\smallskip

\end{itemize}

\end{lemma}

\begin{proof}
We have $\dim |H_{W}| \geqslant 31$ by
Remark~\ref{remark:inequalities}. Suppose that $\dim |H_{W}|
\geqslant 36$. Then it follows from the proof of \cite[Proposition
10.3]{Prokhorov-degree} that $W =
\mathbb{P}(\mathcal{O}_{\mathbb{P}^{2}}(3)\oplus\mathcal{O}_{\mathbb{P}^{2}}(6))$
and the image $\mathbb{P} := \Phi_{\scriptscriptstyle|-K_{W}|}(W)
\subset \mathbb{P}^{38}$ is the anticanonically embedded
threefold $\mathbb{P}(3,1,1,1)$. Hence $X \subset \mathbb{P}^{34}$
is the image of $\mathbb{P}$ under birational linear projection
$\pi : \mathbb{P} \dashrightarrow X$ (see
Remark~\ref{remark:linear-systems}). Let us denote by
$\Omega\subset \mathbb{P}^{38}$ the subspace that $\pi$ projects
from. Then $\dim\Omega = 3$ by \eqref{deg-estimate}.

Now, if $\dim\mathbb{P} \cap \Omega > 1$, then the scheme $Z :=
\mathbb{P} \cap \Omega$ coincides with a quadric in
$\mathbb{P}^3$, which implies that $2 = \deg Z =
K_{\mathbb{P}}^2\cdot Z$. On the other hand, since
$-K_{\mathbb{P}}\sim 2M$ (see Example~\ref{example:examp-1}), we
get $K_{\mathbb{P}}^2\cdot Z\geqslant 4$, a contradiction.

Thus we have $\dim\mathbb{P} \cap \Omega \leqslant 1$ and the
inequality is actually strict. Indeed, if $\mathbb{P} \cap \Omega$
contains a curve, say $C$, then $\deg C \leqslant 4$ (for
$\dim\Omega = 3$ and $\mathbb{P}$ is an intersection of quadrics)
and from the estimate
$$
\mathcal{O}_{\mathbb{P}}(1) \cdot C \leqslant \frac{2}{3}
$$
we obtain that $\text{Sing}(\mathbb{P})\subset C\subseteq
\mathbb{P} \cap \Omega$. This immediately contradicts
Lemma~\ref{theorem:only-cdv-points} because the only singularity
of $\mathbb{P}$ is worse than $\mathrm{cDV}$ (see
Remark~\ref{remark:rem-singularities-of-70}).

Finally, if $\mathbb{P} \cap \Omega$ is a finite scheme, then it
can not be reduced. Indeed, otherwise it does not contain the
singular point of $\mathbb{P}$ (see
Lemma~\ref{theorem:only-cdv-points} and
Remark~\ref{remark:rem-singularities-of-70}), and we get
contradiction via Lemma~\ref{theorem:use-pro-l}. Thus $\mathbb{P}
\cap \Omega$ is a finite non-reduced scheme, and as we have seen
that $\Omega \cap \text{Sing}(\mathbb{P}) = \emptyset$, $\Omega$
coincides with the tangent space at a smooth point on
$\mathbb{P}$. The latter exactly means that $X$ is the threefold
constructed in Proposition~\ref{theorem:simple-case-1}.
\end{proof}

\begin{remark}
\label{remark:minimal-rational-base} Since $V$ is a smooth
rational surface, it follows from \cite[Lemma
10.4]{Prokhorov-degree} that there exists a birational map $W
\dashrightarrow W_{0}$ onto a $\mathbb{P}^{1}$-bundle $W_{0}$ over
a smooth rational surface $V_{0}$ without $(-1)$-curves.
Furthermore, threefold $W_{0}$ and the proper transforms
$\mathcal{L}_{W_{0}}$ and $\mathcal{H}_{W_{0}}$ on $W_{0}$ of the
linear systems $\mathcal{L}_{W}$ and $\mathcal{H}_{W}$,
respectively, possess the same properties as $W$ and
$\mathcal{L}_{W}$, $\mathcal{H}_{W}$ (see
\eqref{some-equality-on-w} and
Remarks~\ref{remark:singularities-of-pairs},
\ref{remark:linear-systems}, \ref{remark:inequalities}). Moreover,
for generic elements $H_{W}\in\mathcal{H}_{W}$ and
$H_{W_{0}}\in\mathcal{H}_{W_{0}}$ the estimate $\dim |H_{W}|
\leqslant \dim |H_{W_{0}}|$ holds. Thus we may assume that either
$V \simeq \mathbb{P}^{2}$ or $V \simeq \mathbb{F}_{n}$ for some $n
\ne 1$. In the latter case, we also have $n \leqslant 4$ (see
\cite[Lemma 10.12]{Prokhorov-degree}).
\end{remark}

Set $c_{i} := c_{i}(\mathcal{E})$ for $i \in \{1,2\}$. Then we get
the following identities on $W$ (cf.
\eqref{can-class-formula-on-w-1}--\eqref{equation:Rieman-Roch-formula-1}):
\begin{equation}
\label{can-class-formula-on-w} -K_{W} = 2H_{W} +
\mathrm{ext}_{\scriptscriptstyle W}^{*}(-K_{V} - c_{1}),
\end{equation}

\begin{equation}
\label{another-hirsch} H_{W}^{2} \equiv H_{W} \cdot
\mathrm{ext}_{\scriptscriptstyle W}^{*}(c_{1}) -
\mathrm{ext}_{\scriptscriptstyle W}^{*}(c_{2}),\qquad H_{W}^{3} =
c_{1}^{2} - c_{2},
\end{equation}

\begin{equation}
\label{equation:formula-for-degree}
\begin{array}{c}
(-K_{W})^{3} = 6K_{V}^{2} + 2c_{1}^{2} - 8c_{2},
\end{array}
\end{equation}

and

\begin{equation}
\label{equation:Rieman-Roch-formula}
\begin{array}{c}
\chi(\mathcal{E}) =
\displaystyle\frac{1}{2}(c_{1}^{2}-2c_{2}-K_{V} \cdot c_{1})+2.
\end{array}
\end{equation}

\begin{proposition}
\label{theorem:P-2-case-1} Let $\dim |H_{W}| \in
\{31,32,33,34,35\}$ and $V = \mathbb{P}^{2}$. Then the vector
bundle $\mathcal{E}$ in \eqref{equation:vector-bundle} is
indecomposable.
\end{proposition}

\begin{proof}
Suppose that $\mathcal{E}$ is decomposable. Then we have
$\mathcal{E} \simeq \mathcal{O}_{\mathbb{P}^{2}}(a) \oplus
\mathcal{O}_{\mathbb{P}^{2}}(a+b)$ for some $b \geqslant 0$. Note
also that $a \geqslant 0$ because the divisor $H_{W}$ is nef.

Further, we have $H^{2i}(\mathbb{P}^{2}, \mathbb{Z}) \simeq
\mathbb{Z}$ for $i \in \{1,2\}$. Hence one may assume that both
$c_i$ are integers. Then we obtain
\begin{equation}
\label{cern-p-2} c_{1} = 2a + b \qquad\mbox{and}\qquad c_{2} =
a^{2} + ab
\end{equation}
by definition of Chern classes (see \cite{hirtz}).

\begin{lemma}
\label{theorem:P-2-case-dec-2} $b \leqslant 3$.
\end{lemma}

\begin{proof}
Suppose that $b > 3$. Pick a generic line $Z$ on $V =
\mathbb{P}^2$ and put $G := \mathrm{ext}_{\scriptscriptstyle
W}^{-1}(Z)$. Notice that $G\simeq\mathbb{F}_b$ by construction.
Then from \eqref{can-class-formula-on-w} and $G^2 \equiv
\mathrm{ext}_{\scriptscriptstyle W}^*(\text{a point})$ we get
$$
K_W\big\vert_G = K_G - G^2 \sim -2h - (b + 3)l
$$
by adjunction. But in this case $K_W\cdot h > 0$, and thus for
varying $Z$ the linear system $|-K_W|$ acquires a fixed component,
which contradicts Remark~\ref{remark:singularities-of-pairs}.
\end{proof}

It follows from \eqref{can-class-formula-on-w} and
Lemma~\ref{theorem:P-2-case-dec-2} that divisor $-K_W$ is (at
least) nef and big for $W =
\mathbb{P}(\mathcal{O}_{\mathbb{P}^2}\oplus\mathcal{O}_{\mathbb{P}^2}(b))$
(we may set $a := 0$ without loss of generality). Then
\eqref{equation:formula-for-degree} and \eqref{formula-for-deg-w}
apply to show that either $b = 3$ or $\dim |-K_W|\leqslant 33$.
But according to Remark~\ref{remark:inequalities} the latter case
is impossible.

Finally, applying \eqref{equation:Rieman-Roch-formula},
\eqref{cern-p-2} and \eqref{equation:dimension-of-H} we may assume
that
$$
\chi(V,\mathcal{E}) =
\frac{1}{2}\left(2a^{2}+2ab+b^{2}+6a+3b\right)+2 \in
\{32,33,34,35,36\}
$$
for $b = 3$. Then, since $c_1\leqslant 9$ by
\cite[10.8]{Prokhorov-degree}, whence $a\leqslant 3$ by
\eqref{cern-p-2}, an elementary computation gives $a=b=3$. But
then $W =
\mathbb{P}(\mathcal{O}_{\mathbb{P}^{2}}(3)\oplus\mathcal{O}_{\mathbb{P}^{2}}(6))$
and the image $\mathbb{P} := \Phi_{\scriptscriptstyle|-K_{W}|}(W)
\subset \mathbb{P}^{38}$ is the anticanonically embedded threefold
$\mathbb{P}(3,1,1,1)$ (see the proof of \cite[Proposition
10.3]{Prokhorov-degree}). This yields $\dim |H_{W}| \geqslant 36$,
a contradiction, and Proposition~\ref{theorem:P-2-case-1} follows.
\end{proof}

\begin{proposition}
\label{theorem:P-2-case-2} Let $\dim |H_{W}| \in
\{31,32,33,34,35\}$. Then $V \ne \mathbb{P}^{2}$.
\end{proposition}

\begin{proof}
Suppose that $V = \mathbb{P}^{2}$. Then the arguments in the proof
of \cite[Proposition 10.3]{Prokhorov-degree} imply that
$\mathcal{E}$ is decomposable if $c_{1} = 9$. Thus, since
$c_1\leqslant 9$ by \cite[10.8]{Prokhorov-degree}, from
Proposition~\ref{theorem:P-2-case-1} we get $c_{1} \leqslant 8$.
(Recall also that $c_1 \geqslant 0$ because $H_W$ is nef.) In this
setting, we get the following result:

\begin{lemma}
\label{theorem:even} The number $c_{1}$ is even.
\end{lemma}

\begin{proof}
Suppose that $c_{1}$ is odd. Then we have
$$
0 \leqslant c_{1} =
2m-3 \leqslant 8
$$
for some integer $2 \leqslant m \leqslant 5$. Further,
\eqref{equation:dimension-of-H} implies that $\chi(V,\mathcal{E})
\in \{32,33,34,35,36\}$. Then from
\eqref{equation:Rieman-Roch-formula} we get
$$
2m^{2} - 3m - c_{2} \geqslant 30.
$$
In particular, we obtain
$$
c_{1}(\mathcal{E}(-m)) = -3, \qquad c_{2}(\mathcal{E}(-m)) = c_{2}
- m^{2} + 3m \leqslant m^{2} - 30 < 0
$$
up to twisting $\mathcal{E}$ by a line bundle (cf. \cite{hirtz}).
Then \eqref{equation:Rieman-Roch-formula} and Serre duality imply
$$
h^{0}(V,\mathcal{E}(-m))+h^{0}(V,\mathcal{E}(-m)\otimes \det
\mathcal{E}(-m)^{*}\otimes \mathcal{O}_{V}(-3)) \geqslant \chi(V,
\mathcal{E}(-m)) \geqslant 1,
$$
and hence $H^{0}(V,\mathcal{E}(-m)) \ne 0$ for $\det
\mathcal{E}(-m)^{*} \simeq \mathcal{O}_{V}(3)$. From this, by
exactly the same arguments as in the proof of \cite[Proposition
10.3]{Prokhorov-degree}, we get contradiction (cf.
\cite[10.11]{Prokhorov-degree}).
\end{proof}

It follows from Lemma~\ref{theorem:even} that
$$
0 \leqslant c_{1} = 2m-2 \leqslant 8
$$
for some integer $1\leqslant m \leqslant 5$. In this case, by
exactly the same arguments as in the proof of
Lemma~\ref{theorem:even}, we find that
$$
c_{1}(\mathcal{E}(-m))=-2, \qquad c_{2}(\mathcal{E}(-m)) = c_{2} -
m^{2} + 2m \leqslant m^{2} + m - 31 < 0
$$
and $H^{0}(V,\mathcal{E}(-m)) \ne 0$. Then again, by exactly the
same arguments as in the proof of \cite[Proposition
10.3]{Prokhorov-degree}, we get contradiction.
Proposition~\ref{theorem:P-2-case-2} is completely proved.
\end{proof}

It follows from Lemma~\ref{theorem:estimate-for-h-64},
Proposition~\ref{theorem:P-2-case-2} and
Remark~\ref{remark:minimal-rational-base} that in order to prove
Theorem~\ref{theorem:main-1} one may reduce to the case when $\dim
|H_{W}| \in \{31,32,33,34,35\}$ and $V = \mathbb{F}_{n}$, where $n
\in \{0,2,3,4\}$. Then we have $H^{4}(\mathbb{F}_{n}, \mathbb{Z})
\simeq \mathbb{Z}$ and $H^{2}(\mathbb{F}_{n}, \mathbb{Z}) \simeq
\mathbb{Z}\cdot h \oplus \mathbb{Z}\cdot l$.

Set $c_{1}: = ah + bl$ and $c_{2}: = c$ for some $a,b,c \in
\mathbb{Z}$.

\begin{proposition}
\label{theorem:very-important-lemma-f-n} The inequalities $0
\leqslant a \leqslant 2$ and $an \leqslant b \leqslant n + 2$
hold.
\end{proposition}

\begin{proof}
The estimates $a \geqslant 0$ and $b \geqslant an$ follow from the
fact that the divisor $H_{W}$ is nef. Further, for generic curve
$Z \sim l$ on $V = \mathbb{F}_n$ and the ruled surface $G :=
\mathrm{ext}_{\scriptscriptstyle W}^{-1}(Z)\simeq\mathbb{F}_m$,
some $m \geqslant 0$, we can write
$$
H_W\big\vert_G \sim \Sigma + (m' + m)L.
$$
Here $\Sigma$ and $L$ are the minimal section and a fiber on $G$,
respectively, and $m' \geqslant 0$ is an integer such that
$\mathcal{E}\big\vert_Z =
\mathcal{O}_{\mathbb{P}^1}(m')\oplus\mathcal{O}_{\mathbb{P}^1}(m'
+ m)$ and $c_1 \cdot Z = a = 2m' + m$. We have the following two
Lemmas:

\begin{lemma}
\label{theorem:very-important-lemma-f-n-ahasa} $m \leqslant 2$.
\end{lemma}

\begin{proof}
Notice that $$K_W\big\vert_G = K_G = -2\Sigma - (m + 2)L$$ due to
adjunction and $G^2 \equiv 0$. Let $m > 2$. Then $K_W\cdot\Sigma >
0$, and thus for varying $Z$ the linear system $|-K_W|$ acquires a
fixed component, which contradicts
Remark~\ref{remark:singularities-of-pairs}.
\end{proof}

\begin{lemma}
\label{theorem:very-important-lemma-f-n-ah} $m' \leqslant 1$ and
$a\leqslant 2$.
\end{lemma}

\begin{proof}
Suppose first that $m' > 0$. Notice that for effective divisor
$D_W \equiv -(K_W + H_W)$ from \eqref{some-equality-on-w}, the
cycle
$$
D_W\big\vert_G \sim \Sigma + (2 - m')L
$$
is effective, since $Z$ is generic. Hence we get $m' \leqslant 2$.

Now, if $m' = 2$, then the sequence of identities
$$
0 = D_W \cdot G^2 = (D_W\big\vert_G)\cdot\Sigma = (\Sigma + (2 -
m')L)\cdot\Sigma = 2 - (m' + m)
$$
yields $m = 0$ as the only option. Hence $D_W$ trivially
intersects all the curves $\sim\Sigma$ on $G \simeq
\mathbb{P}^1\times\mathbb{P}^1$. In particular, since $H_W$ is nef
and big (see Lemma~\ref{theorem:image-of-projection} for the
latter property), $-K_W$ negatively intersects infinitely many
such curves. But this is impossible due to
Remark~\ref{remark:singularities-of-pairs} (or else by
Lemma~\ref{theorem:very-important-lemma-f-n-ahasa}).

Thus we obtain $m' = 1$. But then $m = 0$. Indeed, otherwise
$D_W\cdot\Sigma = 0$ for $m = 1$ and various $G$, so that $-K_W
\cdot \Sigma < 0$ and such $\Sigma$ sweep out a fixed component of
$|-K_W|$, a contradiction. At the same time, when $m = 2$ the
divisor $D_W$ contains an irreducible component $D'_W$, covered by
various $\Sigma\subset G$ (for $D_W\cdot\Sigma<0$). But then,
since both $D_W$ and $D'_W$ are sections of the morphism
$\mathrm{ext}_{\scriptscriptstyle W}$, we find that again $0 = D_W
\cdot G^2 = (D_W\big\vert_G)\cdot\Sigma$, and so $m = 1$ by the
previous arguments when $m' = 2$, hence contradiction.

We conclude that either $m' = 1$, and then $m = 0$, or $m' = 0$.
In the former case, we get $a = 2m' + m = 2$, while in the latter
case we have $a \leqslant 2$ by
Lemma~\ref{theorem:very-important-lemma-f-n-ahasa}.
\end{proof}

Lemma~\ref{theorem:very-important-lemma-f-n-ah} and the estimate
$a \geqslant 0$ prove the first part of
Proposition~\ref{theorem:very-important-lemma-f-n}. Let us finally
show that $b \leqslant n + 2$. Indeed, applying the previous
arguments (and notation) to generic curve $Z \sim  h + nl$, one
immediately gets $m\leqslant n + 2$, where we have used the
identity
$$K_W\big\vert_G = K_G - G^2 = -2\Sigma - (m + n + 2)L$$ on $G = \mathrm{ext}_{\scriptscriptstyle W}^{-1}(Z)\simeq\mathbb{F}_m$, due to $G^2 \equiv \mathrm{ext}_{\scriptscriptstyle
W}^*(n)$ (see \eqref{another-hirsch}) and adjunction (cf. the
proof of Lemma~\ref{theorem:very-important-lemma-f-n-ahasa}).
Moreover, the same argument as in the proof of
Lemma~\ref{theorem:very-important-lemma-f-n-ah} shows that $m' + m
\leqslant n + 2$, i.e., the divisor $D_W\big\vert_G$ is nef. In
this case, switching the roles between $D_W$ and $H_W$ if
necessary, we may assume without loss of generality that $2 + n -
m' \geqslant m + m'$. This implies that $b = 2m' + m \leqslant n +
2$ and concludes the proof of
Proposition~\ref{theorem:very-important-lemma-f-n}.
\end{proof}

Further, let $p,q$ be the integers such that $a = 2p + a',b = 2q +
b'$ for some $a',b' \in \mathbb{Z}$ with $-2 \leqslant a', b'
\leqslant -1$. Consider the twisted vector bundle $\mathcal{E}':=
\mathcal{E}\otimes\mathcal{O}_{\mathbb{F}_n}(-ph-ql)$ and set
$c_{i}':=c_{i}(\mathcal{E}')$ for $i \in \{1,2\}$. Then we get
\begin{equation}
\label{equation:expression-for-c}
\begin{array}{c}
c'_{1} = a'h + b'l, \qquad c_{2}' = c+nap-aq-bp-np^{2}+2pq
\end{array}
\end{equation}
up to twisting $\mathcal{E}$ by a line bundle (cf. \cite{hirtz}).
On the other hand, from \eqref{equation:Rieman-Roch-formula} we
obtain
\begin{equation}
\label{equation:RR-11}
\begin{array}{c}
\chi(V, \mathcal{E}) = -\displaystyle\frac{1}{2}na(a+1)+ab+a+b-c+2
\end{array}
\end{equation}
and
\begin{equation}
\label{equation:RR-2}
\begin{array}{c}
\chi(V, \mathcal{E}') =
(b'-\displaystyle\frac{1}{2}na')(a'+1)+a'-c_{2}'+2.
\end{array}
\end{equation}

\begin{lemma}
\label{theorem:F-0-case} Let $\dim |H_{W}| \in \{31,32,33,34,35\}$
and $V = \mathbb{P}^1\times\mathbb{P}^1$. Then $c_{2}' < 0$ and
$\chi(V, \mathcal{E}')
> 0$.
\end{lemma}

\begin{proof}
It follows from \eqref{equation:dimension-of-H} that
$\chi(V,\mathcal{E}) \in \{32,33,34,35,36\}$. Note that the case
of $\chi(V, \mathcal{E}) > 33$ has been already treated in
\cite[Lemma 6.17]{karz}. From now on let us assume that $\chi(V,
\mathcal{E}) \in \{32,33\}$.

It follows from Proposition~\ref{theorem:very-important-lemma-f-n}
that $0 \leqslant a, b \leqslant 2$. Then from
\eqref{equation:expression-for-c} and \eqref{equation:RR-11} we
get
$$
c_{2}' = \frac{1}{2}ab + a + b - \chi(V, \mathcal{E}) + 2 +
\frac{1}{2}a'b' < 30 - \chi(V, \mathcal{E}) < 0.
$$
Further, we have
$$
\chi(V, \mathcal{E}') = b'(a'+1)+a'-c_{2}'+2
$$
(see \eqref{equation:RR-2}). Then, since $-2 \leqslant a', b'
\leqslant -1$ and $c_{2}' < 0$, we get $\chi(V, \mathcal{E}') >
0$.
\end{proof}

\begin{lemma}
\label{theorem:F-2-case} Let $\dim |H_{W}| \in \{31,32,33,34,35\}$
and $V = \mathbb{F}_{2}$. Then $c_{2}' < 0$ and $\chi(V,
\mathcal{E}') > 0$.
\end{lemma}

\begin{proof}
It follows from \eqref{equation:dimension-of-H} that
$\chi(V,\mathcal{E}) \in \{32,33,34,35,36\}$. Note that the case
of $\chi(V, \mathcal{E}) > 32$ has been already treated in
\cite[Lemma 6.18]{karz}. From now on let us assume that $\chi(V,
\mathcal{E}) = 32$.

It follows from Proposition~\ref{theorem:very-important-lemma-f-n}
that $0 \leqslant a \leqslant 2$ and  $2a \leqslant b \leqslant
4$. Then from \eqref{equation:expression-for-c} and
\eqref{equation:RR-11} we get
$$
c_{2}' \leqslant -\frac{1}{2}a^{2} + \frac{11}{2}a -19 -
\frac{1}{2}a'^{2} + \frac{1}{2}a'b' < 0.
$$
Further, we have
$$
\chi(V, \mathcal{E}') = (b'-a')(a'+1)+a'-c_{2}'+2
$$
(see \eqref{equation:RR-2}). Then $\chi(V, \mathcal{E}') \leqslant
0$ only for $a' = -2$, $b' = -1$ and $c'_{2} = -1$ because $-2
\leqslant a', b' \leqslant -1$ and $c_{2}' < 0$. But in the latter
case we get
$$
-1 = c_{2}' \leqslant -\frac{1}{2}a^{2} + \frac{11}{2}a - 20 < -1
$$
(see \eqref{equation:expression-for-c} and
\eqref{equation:RR-11}), a contradiction. Hence $\chi(V,
\mathcal{E}') > 0$.
\end{proof}

\begin{lemma}
\label{theorem:F-3-case} Let $\dim |H_{W}| \in \{31,32,33,34,35\}$
and $V = \mathbb{F}_{3}$. Then $c_{2}' < 0$ and $\chi(V,
\mathcal{E}') > 0$.
\end{lemma}

\begin{proof}
It follows from \eqref{equation:dimension-of-H} that
$\chi(V,\mathcal{E}) \in \{32,33,34,35,36\}$. Note that the case
of $\chi(V, \mathcal{E}) > 32$ has been already treated in
\cite[Lemma 6.19]{karz}. From now on let us assume that $\chi(V,
\mathcal{E}) = 32$.

It follows from Proposition~\ref{theorem:very-important-lemma-f-n}
that $0 \leqslant a \leqslant 2$ and  $3a \leqslant b \leqslant
5$. Then from \eqref{equation:expression-for-c} and
\eqref{equation:RR-11} we get
\begin{equation}
\label{eq-c-2-yyy} c_{2}' \leqslant -\frac{3}{4}a^{2} +
\frac{13}{2}a - 16 - \frac{3}{4}a'^{2} + \frac{1}{2}a'b' \leqslant
-3.
\end{equation}
Further, we have
$$
\chi(V, \mathcal{E}') = (b'-\frac{3}{2}a')(a'+1)+a'-c_{2}'+2
$$
(see \eqref{equation:RR-2}). Then $\chi(V, \mathcal{E}') \leqslant
0$ only for $a' = -2$ because $-2 \leqslant a',b' \leqslant -1$
and $c_{2}' < 0$. In particular, we have either $\chi(V,
\mathcal{E}') = -1 - c_{2}'$ or $\chi(V, \mathcal{E}') = -2 -
c_{2}'$, and \eqref{eq-c-2-yyy} then implies that $\chi(V,
\mathcal{E}') > 0$.
\end{proof}

\begin{lemma}
\label{theorem:F-4-case} Let $\dim |H_{W}| \in \{31,32,33,34,35\}$
and $V = \mathbb{F}_{4}$. Then $c_{2}' < 0$ and $\chi(V,
\mathcal{E}') > 0$.
\end{lemma}

\begin{proof}
It follows from \eqref{equation:dimension-of-H} that
$\chi(V,\mathcal{E}) \in \{32,33,34,35,36\}$. Note that the case
of $\chi(V, \mathcal{E}) > 32$ has been already treated in
\cite[Lemma 6.20]{karz}. From now on let us assume that $\chi(V,
\mathcal{E}) = 32$.

It follows from Proposition~\ref{theorem:very-important-lemma-f-n}
that $0 \leqslant a \leqslant 2$ and  $4a \leqslant b \leqslant
6$. Then from \eqref{equation:expression-for-c} and
\eqref{equation:RR-11} we get
\begin{equation}
\label{eq-c-2-yyy-1} c_{2}' \leqslant -a^{2} + 7a - 14 - a'^{2} -
a' \leqslant -4.
\end{equation}
Further, we have
$$
\chi(V, \mathcal{E}') = (b'-2a')(a'+1)+a'-c_{2}'+2
$$
(see \eqref{equation:RR-2}). Then $\chi(V, \mathcal{E}') \leqslant
0$ only for $a' = -2$ because $-2 \leqslant a',~b' \leqslant -1$
and $c_{2}' < 0$. In particular, we have either $\chi(V,
\mathcal{E}') = -2 - c_{2}'$ or $\chi(V, \mathcal{E}') = -3 -
c_{2}'$, and \eqref{eq-c-2-yyy-1} then implies that $\chi(V,
\mathcal{E}') > 0$.
\end{proof}

It follows from Lemmas~\ref{theorem:F-0-case},
\ref{theorem:F-2-case}, \ref{theorem:F-3-case} and
\ref{theorem:F-4-case} that to prove Theorem~\ref{theorem:main-1}
it only remains to treat the case when $c_{2}' < 0$ and $\chi(V,
\mathcal{E}') > 0$. But then, given that also $c_1' = a'h + b'l$
for the integers $-2 \leqslant a', b' \leqslant -1$, we apply
Remark~\ref{theorem:no-1-curves-on-y-rem} to get contradiction.

This proves the ``classification'' part of
Theorem~\ref{theorem:main-1}. To obtain the whole statement we
prove the following:

\begin{lemma}
\label{theorem:non-cdv-64} Let $X$ be a Fano threefold as in
Theorem~\ref{theorem:main-1}. If $X \ne \mathbb{P}^{3}$, then the
locus $\mathrm{Sing}(X)$ contains a non-$\mathrm{cDV}$ point,
i.e., singularities of $X$ are worse than $\mathrm{cDV}$.
\end{lemma}

\begin{proof}
It follows from Proposition~\ref{theorem:namikawa-smoothing} and
\cite[Corollary 5.38]{Kollar-Mori} that the cones from
Examples~\ref{example:examp-0}, \ref{example:examp-4} have
singularities worse than $\mathrm{cDV}$. In the remaining cases,
as one can easily see from the proof of
Theorem~\ref{theorem:main-1}, we have a commutative diagram
$$
\xymatrix{
&&X''\ar@{->}[ld]_{\sigma}\ar@{->}[rd]^{\tau}&&\\%
&X'\ar@{-->}[rr]_{\pi}&&X,&}
$$
where

\begin{itemize}

\item $X'$ is $\mathbb{P}(3,1,1,1)$ or $\mathbb{P}(6,4,1,1)$, or
$X_{70}$, or $X_{66}$;

\item $\pi$ is the birational linear projection from a subspace
$\Omega\subset\mathbb{P}^{g'+1}$ (here $g'$ is the genus of $X'$)
passing only through $\mathrm{cDV}$ points on $X'$;

\item $\sigma$ is the restriction to $X'$ of the blowup of the projective space $\mathbb{P}(|-K_{X'}|)$ at $\Omega$;

\item $K_{X''} = \tau^{*}(K_{X})$.

\end{itemize}

In particular, $X'$ and $X''$ are isomorphic near their
non-$\mathrm{cDV}$ points (cf. \cite[Corollary 1.7]{karz}), which
implies that singularities of $X$ are worse than $\mathrm{cDV}$.
\end{proof}

Lemma~\ref{theorem:non-cdv-64} finishes the proof of
Theorem~\ref{theorem:main-1}.

\end{document}